\documentclass[a4paper,11pt,reqno]{amsart}
\usepackage{amsmath, amssymb, amsfonts, amsthm}
\usepackage{latexsym,textcomp}
\usepackage{dsfont}
\usepackage[english]{babel}
\usepackage{hyperref}
\usepackage{enumerate}
\usepackage[utf8]{inputenc}
\usepackage{faktor}

\usepackage{color}
\usepackage[citestyle=alphabetic,
            bibstyle=alphabetic,
            backend=bibtex,
            url=true,
            doi=true,
            isbn=false,
            giveninits=true]{biblatex}

\bibliography{literatur.bib}

\theoremstyle{definition}

\newtheorem{thm}{Theorem}[section]
\newtheorem{exa}{Example}[thm]
\newtheorem{rem}[thm]{Remark}
\newtheorem{defi}[thm]{Definition}
\newtheorem{lemm}[thm]{Lemma}
\newtheorem{propo}[thm]{Proposition}
\newtheorem{corol}[thm]{Corollary}

\numberwithin{equation}{section}
\newcommand{\comment}[1]{}

\newcommand{\cF}{{\mathcal F}}

\newcommand{\cL}{{\mathcal L}}

\newcommand{\cD}{{\mathcal D}}

\newcommand{\Z}{{\mathbb Z}}
\newcommand{\R}{{\mathbb R}}

\newcommand{\N}{{\mathbb N}}

\newcommand{\ran}{{\mathrm {ran}\,}}

\newcommand{\Deg}{{\mathrm {Deg}}}

\newcommand{\al}{{\alpha}}

\newcommand{\ph}{{\varphi}}

\newcommand{\as}[1]{\langle #1\rangle}
\newcommand{\av}[1]{\left\Vert #1\right\Vert}
\newcommand{\aV}[1]{\Vert #1 \Vert}

\newcommand{\aVd}{\left\Vert \cdot\right\Vert}

\newcommand{\Hm}[1]{\leavevmode{\marginpar{\tiny%
$\hbox to 0mm{\hspace*{-0.5mm}$\leftarrow$\hss}%
\vcenter{\vrule depth 0.1mm height 0.1mm width \the\marginparwidth}%
\hbox to 0mm{\hss$\rightarrow$\hspace*{-0.5mm}}$\\\relax\raggedright
#1}}}

\newcommand{\hilbert}{{\mathcal{H}}}

\begin{document}

\title[Graphs at infinity]{Graphs at infinity: Liouville theorems, Recurrence and Characterization of Dirichlet forms}

\author{Matthias Keller}
\address{M. Keller, Institut für Mathematik, Campus Golm, Haus 9, Karl-Liebknecht-Straße 24-25, 14476 Potsdam, Germany}
\email{matthias.keller@uni-potsdam.de}

\author{Daniel Lenz}
\address{D. Lenz, Institut für Mathematik, Friedrich-Schiller-Universität Jena,  Inselplatz 5, 07743 Jena, Germany}
\email{daniel.lenz@uni-jena.de}

\author{Marcel Schmidt}
\address{M. Schmidt,  Mathematisches Institut, Universität Leipzig, Augustusplatz 10, 04109 Leipzig, Germany and Institut für Mathematik, Friedrich-Schiller-Universität Jena,  Inselplatz 5, 07743 Jena, Germany}
\email{marcel.schmidt@math.uni-leipzig.de and schmidt.marcel@uni-jena.de}


\maketitle

\begin{abstract}
 We survey recent results on graphs and their Laplacians related to the behavior of the graph at large. In particular, we focus on Liouville theorems, recurrence and characterizations of Dirichlet forms via boundary terms.
\end{abstract}

%

\setcounter{tocdepth}{1}
\tableofcontents


\renewcommand{\chaptername}{Part}

\section*{Introduction}

Each graph comes with a Laplacian. This Laplacian gives rise to a semigroup and a resolvent. The semigroup  provides  solutions of the heat equation. It can also be described by a Markov process on the graph. The resolvent  provides solution to Poisson equation,  thereby allowing one to study  electrostatics on the graph.  From this very short discussion we can already  infer  how graphs and their Laplacians connect quite a few areas of mathematics.

We also see how  graphs are very similar to Riemannian  manifolds and very similar questions may be investigated for graphs and for manifolds. This is no coincidence. There is a common framework covering Laplacians on manifolds, on  graphs and on many more geometric objects such a fractals and networks. This is the framework of Dirichlet forms. Accordingly, any result formulated in terms of Dirichlet forms   will apply to all these instances.  

In the present article -- which is a survey of results obtained  within the Priority Program ``Geometry at Infinity'' supported by of the German Science Foundation -- we study graphs at infinity via Dirichlet forms.   Of  course, there is an abundance of related  questions   just in the realm of asymptotics of heat kernels alone. Here, we   look in a different direction and consider  
 (non)existence of solutions $u$  to the equation 
$$(\mathcal{L} +1) u =0$$
in various situations,  where $\mathcal{L}$ is the graph Laplacian. Our point of view is that solutions are encoded by their values on some type of boundary at infinity.  To make precise what we mean  by boundary at infinity is part of our task.  A basic idea is that non-existence of certain solutions is related to a completeness property of the graph (giving non existence of the boundary) or to irrelevance of certain boundaries. Conversely, existence of solutions is related to richness of certain boundaries, which encode these solutions.  Specifically, the three parts of the article deal with the following topics: 

The first part deals with non-existence results for solutions in $\ell^p$ for $1\leq p\leq \infty$. We give geometric, completeness type, criteria for non-existence for $1\leq p<\infty$ following \cite{HK14,HKLS22}.  These results are the analogues to famous results on manifolds going back to  Yau \cite{Yau76} and  Karp \cite{Karp82}.  
 We then turn to the case $p =\infty$. This case is of special interest as the non-existence is characteristic for a conservation property of the graph known as stochastic completeness, as studied in many  the works, see the discussion at the end of Section \ref{sec-Grigoryan} for a more detailed account. We  provide a geometric criterion following \cite{HKS20}, which can be seen as analogue to famous results of Grigor'yan \cite{Gri86}.  

The second part then deals with recurrent graphs. Recurrence is characterized by non-existence of positive supersolutions $(\mathcal L + 1)u \geq 0$ to the equation. We discuss a characterization in terms of irrelevance of metric boundaries. Here, irrelevance is measured by means of vanishing  capacity. This is related to results of Yamasaki \cite{Yam75,Yam77} and Soardi \cite{Soa}. We follow \cite{LPS23}.

The third part deals with existence of solutions in $\ell^2$ specifically. This is strongly related to existence of self-adjoint extensions of the Laplacian.   Indeed, the main thrust is to describe certain self-adjoint extensions via boundary-type conditions, where the boundary comes about as a sort of dual object to the $\ell^2$-solutions in question.  {Our discussion here  can be seen as a slightly  specialized   version of our works  \cite{KLSS19, KLSSW23}.  In fact,  both references treat slightly different boundaries, which leads to slightly different statements.  Also, it should be noted, that \cite{KLSSW23} treats general Dirichlet forms, not only the ones arising from graphs on discrete spaces. For these reasons  we  provide rather complete  proofs in the third part (whereas we only sketch arguments in the first two parts). }

Our considerations have counterparts in the case of manifolds and are sometimes even motivated by such counterparts. Also, 
generalizations to arbitrary Dirichlet forms are known in certain cases as well and are under investigation in other cases.  Here, we restrict the present discussion to the graph case for two reasons: Firstly,  the graph Laplacian can be defined with very little technical background. So, the formulation of the problems in question is rather straightforward.  Secondly, for the topics treated below  there are rather neat and complete  answers in the graph case.  However, 
when leaving the graph case,  additional assumptions become necessary depending on the topic in question. 

\setcounter{part}{-1}

\part{Basic concepts and notation}

In this part, we introduce the basic quantities dealt with in this article. These are graphs and the associated forms and operators.
We first  present these quantities in a topological and  measure-free setting and then turn to the situation in which a measure and, hence, an $\ell^2$-space is given. The corresponding spaces will all be spaces of real-valued functions.   The material in this part is well-known, we refer to the book \cite{KLW21} for a thorough treatment.

Throughout this article, $X$ is a countable set equipped with  discrete topology. Then any subset of $X$ is open and  all functions on $X$ are continuous and the compact subsets of $X$ are just the finite sets.  We denote the set of all real-valued  functions on $X$ by $C(X)$. We denote the  set of all functions in $C(X)$, which vanish outside of a compact (i.e.,  finite) set by  $C_c (X)$.  Any subset of $X$ is open and we equip $X$ with the $\sigma$-algebra of all of its subsets. We then consider measures  $m$ with  the property that 
$0< m(x):=m(\{x\}) <\infty$
holds for all $x\in X$.  Whenever $m$ is such a measure, we write $(X,m)$ for the arising measure space {and call $ (X,m) $ a \emph{discrete measure space}}.  We denote by 
$$a \vee b \qquad\mbox{ and }\qquad a\wedge b$$
the  maximum and the minimum, respectively, of two real numbers (or two functions) $a$, $b$. 
{Furthermore, for a set $A$, we denote by $1_A$ its characteristic function.}

\section{Graphs, forms and Laplacians}
{In this section we introduce graphs over $X$, the associated energy form and the formal Laplacian.  We first do this in a topological and measure free setting and then turn to the situation in which a measure and, hence, an $\ell^2$-space is given. The material in this section is well-known, we refer to the book \cite{KLW21} for a thorough treatment.}
\subsection{Graphs}
First, we introduce graphs over countable $X$.

\begin{defi}[Graph over $X$] A  \textit{graph over $X$}
    is  a pair $(b,c)$ consisting of a {symmetric} function $b  \colon X\times X \to [0,\infty)$ {with zero diagonal, i.e.,  $b(x,y)= b(y,x)$ and $b(x,x)= 0$ for all $ x,y\in X $,}   satisfying
    $$ {\mathrm{deg}(x):=}\sum_{y\in X} b(x,y) <\infty $$ for all $x\in X$,
     and a function $c \colon X \to [0 ,\infty)$. The function $b$ is called the \textit{edge weight} and the function $c$ is called the  \textit{killing term}.
     \end{defi}

Whenever $(b,c)$ is a graph over $X$ we refer to   the elements of $X$ 
as \textit{vertices} and to  pairs  $(x,y)$ with $b(x,y)>0$ as \textit{edges}. By symmetry of $b$ a pair $(x,y)$ is an edge if and only if $(y,x)$ is an edge. 
We   call $x$ and $y$ \textit{neighbours} and write $x\sim y$ if $(x,y)$ is an edge. 
A \emph{path}  is a (finite or infinite) sequence
$(x_1,x_2,\ldots)$ of vertices such that $x_i \sim x_{i+1}$, for $i =
1,2,\ldots$. We say that two points $x,y\in X$ are {\em connected}
if there is a finite path $(x=x_1,\ldots,x_n=y)$. This defines an
equivalence relation on the set of vertices and the resulting
equivalence classes are called \emph{connected components}. 
We say that a graph is \emph{locally finite} if for all $x\in X$ the
{\em set of its neighbors} $\{y \in X\mid x\sim y\}$ is finite.

We call $(b,c)$ a graph over $(X,m)$ if  $X$ is equipped with a measure $m$ as above.

\begin{rem} 
(a)  The reader may wonder about the $c$ in the definition of graphs. Indeed, in the literature one can also  find the notion of graph defined without such  $c$.  For our structural considerations the $c$ is  important: With our definition one can show a one-to-one correspondence between graphs and regular Dirichlet forms (an analytic concept) and likewise between graphs and symmetric Markov processes on $X$ with certain regularity properties (a stochastic concept). On a more direct level  the $c$ can easily be seen to be necessary for  restrictions of graph  Laplacians  to subsets to be graph Laplacians again (see below for definitions).

(b)  In our discussion we have assumed
that $X$ is countable. This is not necessary in order to set up the
theory. Indeed, all of the preceding definitions make sense also for
uncountable $X$. However, the summability condition $\sum_{y\in X} b(x,y)
<\infty$ for all $x\in X$ implies that any $x\in X$ can have at most
countably many neighbors.   From this we can then deduce that any connected component must be countable. So, 
even  if we started with an uncountable $X$,  we would end up with an (uncountable) union of disjoint countable graphs, each of which could be treated by the methods presented below. 
For this reason we just assume
countability of $X$ from the very beginning.
\end{rem}

\subsection{The energy form}
Any graph comes with a bilinear map on (a subspace of) $C(X)$. This map underlies all our subsequent considerations.

To a graph $(b,c)$ over $X$, we  associate the map
$$\mathcal{Q} =\mathcal{Q}_{b,c} \colon C(X)\to [0,\infty]$$
defined by
$$\mathcal{Q}(f) := \frac{1 }{2 } \sum_{x,y \in X} b(x,y) (f(x) - f(y))^2  + \sum_{x \in X} c(x) f(x)^2.$$
We call $Q$ the \textit{energy form} associated to $(b,c)$.  Then, $$\mathcal{D} =\mathcal{D}_{b,c}:=\{ f\in C(X) \mid \mathcal{Q}(f)<\infty\}$$
is a subspace of $C(X)$ containing $C_c (X)$. The elements of $\mathcal{D}$ are called \emph{functions of finite energy}.   
   Fatou's lemma easily gives that the energy form is lower semicontinuous, i.e.,  if a sequence $(f_n)$ in $C(X)$ converges
    pointwise to $f\in C(X)$, then
 $$\mathcal{Q} (f) \leq \liminf_{n\to\infty} \mathcal{Q} (f_n)$$
    holds.  Recall that a  map $ C \colon \R \to \R  $ with
		$$ C(0) =0 \quad\mbox{ and }\quad |C(u) - C(w)|\leq |u - w|$$
	 for all $u,w\in \R$ is called \textit{normal contraction}. Modulus and positive and negative part of a number  are examples for normal contractions. Now, a crucial feature of $\mathcal{Q}$ is that it   is  compatible with normal contractions in that 
	$$\mathcal{Q}(Cf)\leq \mathcal{Q}(f)$$
	holds for any $f\in C(X)$ and any normal contraction $C$, where we use the notation $Cf = C  \circ f$. Indeed, this follows by a direct computation.

The map $\mathcal{Q}$ satisfies the parallelogram identity. By polarization we can therefore lift it to the bilinear map    							
$$\mathcal{Q} \colon \mathcal{D}\times \mathcal{D} \to \R$$
given  by 
$$\mathcal{Q}(f,g)=\frac{1 }{2 } \sum_{x,y \in X} b(x,y)  {(f(x) - f(y))} (g(x)- g(y)) + \sum_{x \in X} c(x) {f(x)} g(x).$$
Clearly, $\mathcal Q$ is symmetric and, for all $f \in \mathcal D$, we have  
$$\mathcal{Q}(f,f) =\mathcal{Q}(f) \geq 0.$$

\subsection{The formal  Laplacian}
Besides the energy form $\mathcal{Q}$ associated to a graph we will
also consider the \textit{formal Laplacian}. {To this end, we fix a discrete measure space $(X,m)$ and consider graphs over $(X,m)$.}

Let $(b,c)$ be  graph over $(X,m)$. We define the operator $\mathcal{L}{:=\mathcal{L}_{b,c}} \colon \cF \to C(X)$ on
$$
\mathcal{F} = \mathcal{F}_{b} := \{ f \in C(X)  \mid \sum_{y\in X} b(x,y) | f(y)|<\infty \; \mbox{ for all }   x\in X  \}$$
by
$$\mathcal{L}  f(x) :=  \frac{1}{m(x)}\sum_{y\in X} b(x,y) (f(x) - f(y)) + \frac{c(x)}{m(x)} f(x).$$
We call $\mathcal{L}$ the \textit{formal Laplacian} associated to $(b,c)$ over $(X,m)$. The word ``formal'' is used as this is not an operator in an $\ell^2$ space.

The  operator $\mathcal{L}$ has a certain symmetry property and the form $\mathcal{Q}$ and operator $\mathcal{L}$ are related by an  integration by parts formula, which we refer to as Green's formula: Specifically,  it is easy to see that
$$C_c (X)\subset \mathcal{D}\subset \mathcal{F}$$
holds and,  for all $f\in \mathcal{D}$  and $\varphi\in C_c (X)$, the equality
$$\mathcal{Q} (\varphi,f)  = \sum_{x\in X}  {\varphi(x)} \mathcal{L} f (x)m(x) = \sum_{x\in X} \mathcal{L} \varphi (x) f (x)m(x)$$
is valid.
Green's formula is even valid in the following more general form:  For $f\in \mathcal{F}$ and $\varphi \in C_c (X)$, we have
        \begin{multline*}
             \frac{1}{2}\sum_{x,y\in X} b(x,y) (\varphi (x) - \varphi (y)) (f(x) -f (y)) + \sum_{x\in X} c(x)  \varphi (x) f(x)\\
              =\sum_{x\in X}  \varphi (x) \mathcal{L} f (x)m(x)  = \sum_{x\in X} \mathcal{L} \varphi (x) f(x)m(x),
        \end{multline*}         where all of the sums are absolutely convergent.

We now turn to  two fundamental equations involving the Laplacian: The \textit{Poisson equation} has the form 
$$(\mathcal{L} + \alpha) u = f,$$
where $\alpha \geq 0$ and $f\in C(X)$ are given and $u\in \mathcal{F}$ is the looked for solution. The case $f =0$ is particularly relevant. 
We say that  $u$ is $\al$-\emph{harmonic} if $u \in \mathcal{F}$ and it satisfies
$$(\mathcal{L}+\al)u = 0.$$
When $\al=0$, we say that $u$ is \emph{harmonic}.

The \textit{heat equation} reads 
$$(\mathcal{L}+\partial_t)u=0.$$
Specifically, a function $ u\colon[0,\infty)\times X
\to \R $ is called a \emph{solution of the heat
equation} if, for every $ x\in X $, the mapping $ t\mapsto u_{t}(x)
$ is continuous on $ [0,\infty) $ and differentiable on $ (0,\infty)
$, $ u_{t} \in \mathcal{F} $ for all $ t>0 $,  and
$$(\mathcal{L}+\partial_{t})u_{t}(x)=0$$
for all $x \in X$ and $t>0$.  If $u$ is a solution of
the heat equation and $u_0=f$ for $f \in C(X)$, then $f$ is called
the \emph{initial condition} for $u$. We will say that $u$ satisfies
the heat equation with initial condition $f$ in this case. We think
of $x$ as a space variable and $t$ as time.

We are going to treat heat equation and Poisson equation via operators on Hilbert space. In order to do so we will have to provide Hilbert spaces and also recall some abstract theory of forms and operators. We recall the abstract theory in the next section and subsequently then deal with  the Hilbert spaces in our context. These Hilbert spaces have the form $\ell^2 (X,m)$.

\section{Dirichlet forms and their  generators}
In this section we briefly recall the abstract theory of Dirichlet forms and their operators. For the abstract theory, we refer to \cite{FOT} and in a second step we focus on the special case of Dirichlet forms on discrete measure space which are most relevant here, see \cite{KLW21} for a recent monograph.
\subsection{Abstract framework}
Forms and positive operators can be seen as two faces of the same medal. The advantage of forms is that they have  a larger domain of definition and are generally much more easily written down than the underlying operator.

Assume that $\hilbert$ is a real Hilbert space. An operator $L$ on $\hilbert$ is a linear map from a subspace $D(L)$ of $\hilbert$, called the \textit{domain} of $L$,   to $\hilbert$. If $D(L)$ is dense, then the adjoint  $L^*$ of $L$ is the operator with domain $D(L^*) $ given by
$$\{ f\in \hilbert \mid \mbox{there exists $g\in \hilbert$ with $\langle f,  L h\rangle = \langle g, h\rangle$ for all $h\in D(L)$} \}$$
and acting by 
$$L^* f = g.$$
The operator $L$ is called \textit{self-adjoint} if $D(L)$ is dense and $L = L^*$ holds. A self-adjoint operator is called \textit{positive} if $\langle L f,f\rangle \geq 0$ holds for all $f\in D(L)$. We write $L\geq 0$ whenever $L$ is a positive operator.
    
		A \emph{symmetric positive form}  $Q$ on  the Hilbert space $\hilbert$
    consists of a subspace $D(Q) \subset \hilbert$ called the \emph{domain
        of} $Q$ together with a  map
    $$Q \colon  D(Q) \times D(Q) \to  \R$$
    satisfying
    \begin{enumerate}

        \item[$\bullet$] $Q(f,g) = {Q(g,f)}$ \hfill $ ($\emph{``Symmetry''}$ ) $

        \item[$\bullet$] $Q(f,\alpha g + \beta h) = \alpha Q (f,g) + \beta Q(f,h)$ \hfill $ ($\emph{``Linearity''}$ )$ 

        \item[$\bullet$] $Q(f,f)\geq 0$ \hfill $ ($\emph{``Positivity''}$ ) $
    \end{enumerate}
    for all $f, g, h \in D(Q)$ and $\alpha, \beta \in  \R$. It is called {\em densely defined} if $D(Q)$ is dense in $\hilbert$. We will often refer to positive symmetric forms as just forms. Whenever such a form $Q$ is given, we define, for $f \in \hilbert$, the value
$$Q'(f) := \begin{cases}
           Q(f, f)  & \textup{if } f\in D(Q)\\
           \infty  & \textup{otherwise}.\\
           \end{cases}
$$
We note that we can recover the form $Q$ from the values $Q'(f)$ for $f \in \hilbert$ as the domain of $Q$ is given by
$$D(Q) =\{ f\in\hilbert \mid Q' (f) < \infty\}$$
and $Q(f,g)$ can be obtained by using the polarization identity. For this reason we will subsequently  write $Q(f)$ instead of $Q'(f)$. This viewpoint  lets  us compare the size of two forms $Q, \widetilde Q$ on $\hilbert$, even though they may live on different domains. More precisely, we write
$$Q \leq \widetilde Q$$
if $Q(f) \leq \widetilde Q(f)$ for all $f \in \hilbert$. Spelled out, this means $D(\widetilde Q) \subset D(Q)$ and $Q(f) \leq \widetilde Q(f)$ for all $f \in D(\widetilde Q)$.

Whenever $Q$ is a form, then 
$$Q_1:= Q +\langle \cdot,\cdot \rangle$$
is an inner product.  The associated norm is referred to as \textit{form norm} and denoted by $\av{\cdot}_Q$. The form is called \textit{closed} if $D(Q)$ is a complete with respect to the form norm. Equivalently, the form is closed if $D(Q)$ is   a Hilbert space  with respect to this inner product $Q_1$.  It is a fundamental result that the  form $Q$ is closed if and only if $Q'$ is lower semicontinuous on $\hilbert$.

There is a one-to-one correspondence between densely defined closed forms  and positive operators: To any densely defined closed form $Q$ there exists a unique self-adjoint operator $L=L_Q\geq 0$ with
$$Q(f,g) =\langle f, Lg\rangle$$
for all $f\in D(Q)$ and $g\in D(L)$.  Conversely, whenever a positive self-adjoint $L$ is given, using functional calculus we obtain a densely defined form $Q_L$ via $D(Q_L) = D(L^{1/2})$ and
$$Q_L (f,g) = \langle L^{1/2} f, L^{1/2} g\rangle.$$
The maps $L\mapsto Q_L$ and $Q\mapsto L_Q$ are inverse to each other. 

Any positive operator or, equivalently, any form comes with two canonical families of operators:   We denote the set of bounded linear   operators on the Hilbert space $\mathcal{H}$ by $B(\mathcal{H})$. 

The map $$T \colon [0,\infty)\to B(\mathcal{H}),\quad t\mapsto e^{-t L}, $$
satisfies 
$$T_0 = \mathrm{id} \quad\mbox{ and } \quad T_{t+s} =T_t \circ T_s$$
for all $t,s\geq 0$.  The family $(T_t)_{t \geq 0}$ is referred to as the \textit{semigroup} associated to $L$ (or $Q$).  For any $f\in D(L)$, the function $u: [0,\infty)\to \hilbert$, $  t\mapsto T_t f$ is the unique solution of the heat type  equation
$$(L+\partial_{t})u=0,\qquad
 u(0) =f.   $$
The map $$G \colon(0,\infty)\to B(\mathcal{H}),\quad  G_\alpha = (L+\alpha)^{-1},$$
satisfies
$$G_\alpha - G_\beta = (\beta - \alpha) G_\alpha G_\beta$$ for all $\alpha,\beta>0$. The family $(G_\alpha)_{\alpha > 0}$ is referred to as 
\textit{resolvent} of $L$ (or $Q$). For any $f\in \hilbert $ the vector $G_\alpha f$ is the unique solution oft the Poisson type equation  $$ (L+\alpha) u = f. $$

So far we have considered forms and operators on an arbitrary Hilbert space. Now, we assume that the Hilbert space $\mathcal{H}$ is an $L^2$-space (as e.g. $\ell^2 (X,m)$  above). Then,  in order to really   qualify as heat equation or Poisson equation, we want our semigroup and resolvent to map $f$ with $0\leq f\leq 1$ to $g$ with $0\leq g\leq 1$.  This property has a name on its own:  A bounded operator $A$ is called \textit{Markovian} if $$ 0\leq A f\leq 1 $$ holds for all $f$ with  $0\leq f\leq 1$. The forms giving rise to semigroups of Markovian operators have a characteristic feature and are introduced next.

\begin{defi}[Dirichlet form] A form on an $L^2$-space is called {\em Markovian} if it is compatible with all normal contractions (i.e.,  satisfies $Q(Cf)\leq Q(f)$ for all $f$ in the $L^2$-space and all normal contractions $C$). A closed densely defined Markovian form $Q$ on an $L^2$-space is called a \emph{Dirichlet form}.
\end{defi}

 The crucial  result is that Dirichlet forms are are exactly the forms whose semigroups and resolvents consist of Markovian operators.
 
\begin{thm}[Characterization of Dirichlet forms] Let $Q$ be a closed form on an $L^2$-space with associated semigroup $(T_t)_t$ and resolvent $(G_\alpha)$. Then, the following assertions are equivalent:
\begin{enumerate}[(i)]
\item $Q$ is a Dirichlet form.
\item $T_t$ is Markovian for each $t\geq 0$.
\item $\alpha G_\alpha$ is Markovian for each $\alpha >0$.
\end{enumerate}
\end{thm}

It turns out that Dirichlet forms can not only be characterized by compatibility with all normal contractions but compatibility with certain contractions suffices. For example, a closed form $Q$ is a Dirichlet form if $Q(f\wedge 1) \leq Q(f)$ holds for all $f$ in the domain of $Q$.

\subsection{Dirichlet forms on graphs}
In this section we use the abstract form theory presented in the last section in order to introduce two forms associated to a graph over $X$ provided $X$ is equipped with a measure $m$.  The relevant Hilbert space will be $\ell^2 (X,m)$.

Let a graph $(b,c)$ over $(X,m)$ be given.  Let $\mathcal{Q}$ be the associated energy form and $\mathcal{D}$ its domain. We are interested in closed forms $Q$ that are restrictions of $\mathcal{Q}$ and whose domain contains $C_c (X)$.  There are two natural candidates for such restrictions, which can be thought of as having the  maximal one and the minimal possible domain. 

The maximal possible domain  for a restriction of $\mathcal{Q}$ to $\ell^2(X,m)$ is clearly $\mathcal{D}\cap \ell^2 (X,m)$. We define the associated form $Q^{(N)} = Q^{(N)}_{b,c}$ by 
$$D ({Q}^{(N)} ) := \mathcal{D} \cap \ell^2  (X,m) \quad\mbox{ and }\quad {Q}^{(N)} (f,g) :=\mathcal{Q}(f,g)$$
for $f,g \in D({Q}^{(N)})$.  Then, clearly ${Q}^{(N)}$ is symmetric and positive as $\mathcal{Q}$ has these properties.  We think of
$Q^{(N)}$ as arising from some sort of Neumann boundary conditions and this is the reason for the superscript $(N)$.  We will refer to $Q^{(N)}$ as
the \emph{Neumann form}. From lower semicontinuity of $\mathcal{Q}$ we infer lower semicontinuity of $Q^{(N)}$, see e.g. \cite{RSI,KLW21}. Therefore $Q^{(N)}$ is a closed from.  Moreover, it turns out immediately that $Q^{(N)}$ is a Dirichlet form.

The associated operator is denoted by $L^{(N)}$ and referred to as \emph{Laplacian with Neumann boundary conditions}. At this point this is just a notation, as we have not yet introduced a boundary on which boundary conditions may be considered. We will later see that it can be given sense. From Green's formula it can easily be derived that $L^{(N)}$ is a restriction of the formal Laplacian $\mathcal{L}$, i.e.,
$$L^{(N)} f = \mathcal{L} f$$
holds for all $f\in D(L^{(N)})$.

By construction $Q^{(N)}$ is  maximal in terms of having the biggest possible domain. Similarly, there is a minimal form. This form comes about by considering all symmetric closed forms which are restrictions of $Q^{(N)}$ (or $\mathcal{Q}$)  and whose domain contains $C_c (X)$. The intersection over the domains of all such forms will be a form norm closed subspace of $D(Q^{(N)})$. Hence, the restriction of $\mathcal{Q}$ to this domain will yield a positive closed form. We  denote this form  by $Q^{(D)} := Q^{(D)}_{b,c}$ and its domain by $D (Q^{(D)}) := D(Q^{(D)}_{b,c})$. By construction, $Q^{(D)}$ is the smallest closed form extending the restriction of $\mathcal{Q}$ to $C_c (X) \times C_c (X)$. We can also obtain $D(Q^{(D)})$ by taking the form norm closure of $C_c(X)$. We think of $Q^{(D)}$ as arising from some sort of Dirichlet boundary conditions and this is the reason for the superscript $(D)$. By construction the form $Q^{(D)}$ is closed. Moreover,  it turns out that $Q^{(D)}$ is a Dirichlet form, \cite[Lemma~1.16]{KLW21}.  We denote the associated self-adjoint operator by $L^{(D)}$ and refer to as Dirichlet Laplacian.

Again, at this point this is purely formal piece of notation. We will later give it sense. As above,  Green's formula  implies that $L^{(D)}$ is a restriction of the formal Laplacian $\mathcal{L}$, i.e.,
$$L^{(D)} f = \mathcal{L} f$$
holds for all $f\in D(L^{(D)})$. Indeed, with the same argument, one obtains that the self-adjoint operator of a closed form that is a restriction of $Q^{(N)}$ and an extension of $Q^{(D)}$ is a restriction of $\mathcal L$.

\begin{rem} (a) It is possible for  $Q^{(N)}$ and $Q^{(D)}$ to  agree. This is the case if and only if $\mathcal{Q}$ has only one closed restriction to $\ell^2 (X,m)$ whose domain contains $C_c (X)$. Quite a bit of research has been devoted to giving criteria for this to happen.

For example, both forms agree if $\inf_{x\in X}  m(x) >0$  (e.g. if $m =1$)  or if there exists a $\kappa >0$ such that  $\sum_{y \in X} b(x,y) + c(x)\leq \kappa m(x)$ for all $x\in X$, see e.g. \cite{KLW21}.

In general, however, $Q^{(N)}$ and $Q^{(D)}$ will be different. In the last part of this article  We will have much to say on how possible Dirichlet forms between $Q^{(D)}$ and $Q^{(N)}$ can be described via boundaries (and on what is actually meant by Dirichlet form between $Q^{(D)}$ and $Q^{(N)}$).

(b) As discussed above both $L^{(D)}$ and $L^{(N)}$ are restrictions of $\mathcal{L}$. In general, it is hard to describe explicitly their domains. If $\mathcal{Q}$ possesses only one closed restriction (see (a) of this remark), then $L^{(D)}$ and $L^{(N)}$ agree and their domain is the maximal one, i.e., consists of those $f \in D(Q^{(N)})$ with $\mathcal{L} f \in \ell^2 (X,m)$.
\end{rem}

Let $\av{\cdot}_\infty$ denote the supremum norm on $C_c(X)$. A Dirichlet form $Q$ over $(X,m)$ is called \textit{regular}\index{form!regular Dirichlet} if $D(Q)\cap C_c (X)$ is dense in both $C_c (X)$ with respect to $\av{\cdot}_\infty$ and in  $D(Q)$ with respect
to the form norm $\av{\cdot}_Q$. It turns out that a Dirichlet form $Q$ on $(X,m)$ is regular if and only if $C_c(X)\subset D(Q)$ and $Q$ is the closure of the restriction of $Q$ to the subspace $C_c(X)$. A further investigation then yields the following result, see \cite[Theorem~1.18]{KLW21}. 

\begin{thm}[Regular Dirichlet forms and graphs]\label{IG:t:graphs_regular_df}
    The map $$(b,c) \mapsto Q^{(D)}_{b,c,m}$$ is a bijective correspondence
    between graphs $(b,c)$ over $(X,m)$ and regular Dirichlet forms over $(X,m)$.
\end{thm}
\begin{rem} Regular Dirichlet forms are in a one-to-one correspondence with symmetric Markov processes with suitable regularity properties. As already alluded to  above it is due to these correspondences between graphs, Dirichlet forms and Markov processes  that we include the killing term $c$ in our definition of graphs. 
\end{rem}

\section{Intrinsic metrics}\label{sec-intrinsic}
In this section we introduce the (pseudo)metrics relevant for our
considerations and discuss their properties.

A symmetric function $\sigma  \colon X \times X \to [0,\infty)$ is called a \emph{pseudometric} if it satisfies the triangle inequality, i.e.,
if for all $x,y,z \in X$ it satisfies
$$\sigma(x,y) \leq \sigma(x,z) + \sigma(z,y).$$
For $r \geq 0$ and $x \in X$ we denote the corresponding ball of radius $r$ around $x$ by
$${B_r(x) :=}B_r^\sigma(x) := \{y \in X \mid \sigma(x,y) \leq r\}.$$
For $x \in X$ the distance $\sigma_U$ from a nonempty subset $U\subset X$ is defined by
$$\sigma_U(x) := \sigma(x, U) := \inf\limits_{y\in U} \sigma(x,y)$$
and the \emph{diameter} of $U$ with respect to $\sigma$ is
$$\mathrm{diam}_\sigma(U) := \sup\limits_{x,y \in U} \sigma(x,y).$$
A function $f$ on $X$ is called \textit{Lipschitz-function} with respect to the pseudometric $\sigma$ if there exists a $C\geq 0$ with
$$|f(x) - f(y)|\leq C \sigma(x,y)$$ 
for all $x,y\in X$. We then also say that $f$ is a $C$-Lipschitz function.  The set of all Lipschitz-functions with respect to $\sigma$ is denoted by ${\rm Lip}_\sigma(X)$.

For a graph $(b,c)$ over $X$,  a pseudometric $\sigma$ is called \emph{intrinsic with respect to the measure $m$} if, for all $x \in X$, it satisfies
$$\frac{1}{2}\sum_{y \in X} b(x,y)\sigma(x,y)^2 \leq m(x).$$
We note that it depends on both the graph and the measure whether a given pseudometric on $X$  is intrinsic. This condition can be seen as a discrete version of $|\nabla \sigma(x,\cdot)| \leq 1$, $x \in X$, see the discussion at the end of this section.

Clearly, a pseudometric $\sigma$ is intrinsic with respect to some finite measure  if and only if
$$\sum_{x,y \in X} b(x,y)\sigma(x,y)^2 <\infty$$
holds. This case is of particular interest for us.  Indeed, there are strong ties  between functions of finite Dirichlet energy and intrinsic pseudometrics with respect to a finite measure {in the case \begin{align*}
c=0.
\end{align*}} The Lipschitz functions with respect to such an intrinsic metric are functions of finite energy and, conversely, functions of finite energy induce such intrinsic metrics. Details are given next.

Whenever   $\sigma$ is   an intrinsic pseudometric with respect to the finite measure $m$ on $X$, then  any function $f$ that is  $C$-Lipschitz with respect to $\sigma$  and constant  on  $U \subset X$  satisfies {whenever 
$c=0$}
$$Q(f)\leq  C^2 \min\{m(X),2 m(X\setminus U)\}.$$
In particular,  with $U = \emptyset$, we obtain  the inequality  $$Q(f)\leq C^2 m(X)$$  for any $C$-Lipschitz function $f$
with respect to $\sigma$ so that  Lipschitz function with
respect to $\sigma$ belong to $\mathcal{D}$. Moreover, we also infer 
the inequality
 $$ \mathcal{Q}(\sigma_U) \leq \min\{m(X), 2 m(X\setminus U)\} $$
 and, $\sigma_U $ belongs to $ \mathcal{D}$. Conversely,  for any function $f$ of finite energy,
 the function
 $$\sigma_f \colon X\times X \rightarrow [0,\infty), \quad \sigma_f(x,y) := |f(x) - f(y)|$$
 is an intrinsic pseudometric with respect to the finite  measure $m_f$  given by
$$m_f(x) = \frac{1}{2}\sum\limits_{y\in X} b(x,y)\sigma_f(x,y)^2$$
and the function $f$ is  $1$-Lipschitz  with respect to $\sigma_f$ with  $m_f(X) =   \mathcal{Q}(f)$. In general, this $\sigma_f$ will not be a metric (as values of $f$ in different points need not be distinct). However, this can easily be achieved by an arbitrarily small perturbation. So, when dealing with intrinsic pseudometrics with respect to a finite measure we can always assume without loss of generality that these are actually metrics rather than just pseudometrics. Even more is true,  there  always exists an intrinsic metric with respect to some  finite measure  that induces the discrete topology.

For us   (intrinsic) pseudometrics coming from paths will play a special role. These are discussed next.  Given a symmetric function $w \colon X \times X \to [0,\infty)$ and a (possibly infinite) path $\gamma = (x_1,x_2,\ldots)$ in the graph, we define the length of $\gamma$ with respect to $w$ by
$$L_w (\gamma) := \sum_i w(x_i,x_{i+1}) \in [0,\infty].$$
If the graph is connected, this induces the \emph{path pseudometric} $d_w$ on $X$  via
$$d_w(x,y) = \inf \{L_w(\gamma) \mid \gamma \text{ is  a path  from $x$ to $y$}\}.  $$
We say that $\sigma$ is a {\em path pseudometric} on $X$ if $\sigma = d_w$ for some symmetric function  $w$. 
A symmetric  function $w$ is called \emph{adapted} to  the measure $m$ if for all $x \in X$ it satisfies
$$\frac{1}{2}\sum_{y \in X} b(x,y)w(x,y)^2 \leq m(x).$$
The path metric $d_w$ associated to an adapted $w$ is intrinsic with respect to the measure. 

While not  every intrinsic metric is a path metric, to every intrinsic metric there exists a path metric with the same distance between neighbors. Specifically, whenever $\sigma$ is an intrinsic metric, then $\sigma$ is  adapted. Hence, $d_\sigma$ is then an intrinsic metric  and $d_\sigma (x,y) = \sigma (x,y)$ can easily be established for $x,y\in X$ with $x\sim y$ {while we only have $ d_{\sigma}\le \sigma $ in general}. For this reason, intrinsic metrics can often be chosen to be path metrics.

\begin{rem}[Background on intrinsic metrics]
Intrinsic metrics have long proven to be a useful tool in spectral 
geometry of manifolds and, more generally, for strongly local
Dirichlet spaces, see e.g. Sturm's seminal work \cite{Stu94,Stu2}.
For general Dirichlet spaces, including graphs, a systematic
approach was developed in \cite{FLW}. A key point in \cite{FLW} is a
Rademacher type theorem. In the  context of graphs  this theorem
says that a pseudometric $\sigma$ is intrinsic if and only if for
all $1$-Lipschitz functions $f \colon X \to \R$ with respect to
$\sigma$, we have $|\nabla f|^2 \leq 1$. Here, for  $f \in C(X)$ and
$x \in X$, the quantity
 $$|\nabla f|^2 (x): = \frac{1}{2m(x)}\sum_{y \in X} b(x,y)(f(x) - f(y))^2$$
can be interpreted as the square of the norm of the discrete
gradient of $f$ at $x$ (with respect to the measure $m$). For graphs
with measure $m$ for which the scaled degree $\deg /m$ is uniformly bounded the combinatorial metric (the path metric with respect to the weight $w = 1$)  is an
intrinsic metric (up to a constant). For graphs with unbounded {scaled} degree
this is not the case anymore. For such graphs,  intrinsic metrics
(rather than the combinatorial metric) have  turned out to be the
right metrics for  various questions, see e.g. the survey
\cite{Kel}. 
\end{rem}

\part{Liouville theorems on $\ell^p$}

The classic Liouville theorem states that every bounded harmonic function is
constant. We will present  variants of such a result, each
arising  by replacing the assumption of boundedness  by an $
\ell^{p} $ bound together with suitable assumptions on the geometry.  We consider the two cases $1\leq p <\infty$ and $p =\infty$ separately.   The case $p =\infty$ is related to a feature known as stochastic completeness.

Our results below  will rely on some control of the geometry at large. This control will be expressed by features of  of balls with respect to intrinsic metrics.  We will be particularly interested in intrinsic metrics that generate the discrete topology and entail the Heine-Borel-property that any closed ball is compact. It is not hard to see that a metric on a discrete set induces the discrete topology and gives the Heine-Borel-property if and only if all balls of finite radius with respect to this metric are finite. 

Finiteness of balls will be a crucial assumption in our treatment of $p = \infty$. For the case $p =2$ (or more generally $1<p <\infty$), a weaker property suffices: We will assume that  the  \textit{weighted degree} 
$$\mbox{Deg} : X\to [0,\infty)$$
for the graph $(b,c)$ over $(X,m)$ defined by 
  by $$\mbox{Deg} (x) :=\frac{1}{m(x)}\left(\sum_{y\in X} b(x,y) + c(x)\right)$$
is bounded on all balls of finite radius.

\section{The case  $1\leq p <\infty$}\label{sec-Yau-Karp}
As as warm-up we give a simple result on triviality of subharmonic functions in $\ell^p$ provided that all infinite paths have infinite mass. 

\begin{thm}[Criterion for paths with infinite mass]
Let $(b,c)$ be a connected infinite graph over  $(X,m)$. Assume that every infinite  path has infinite measure, i.e.,
$$\sum_{n=1}^{\infty} m(x_n)=\infty$$
for every sequence of pairwise distinct vertices $x_n$ with $b(x_n,x_{n+1})>0$, $n\in\N$. 
Let $\al\geq 0$ and $u\in \mathcal{F}$ with $u\geq 0$ satisfy
\begin{align*}
(\mathcal{L}+\al)u \leq 0.
\end{align*}
Then $u = 0 $ holds if $ u $ belongs to $\ell^p (X,m)$ for some $p \in [1,\infty)$.
\end{thm}
\begin{proof}  If $u$ is constant, the  condition on paths easily  implies $u =0$. Thus, we now  consider non-constant $u$. For such $u$ there must exist $x_{0}, x_{1}\in X$ with $x_0\sim x_1$  and $u(x_1) > u(x_0) \geq 0$. Now, if $u(x_1) \geq u(y)$ for all $y \sim x_1$, then
$$
(\mathcal{L}+\al)u(x_1) = \frac{1}{m(x_1)}\sum_{y \in X}b(x_1,y)(u(x_1)-u(y))+\left(\frac{c(x_1)}{m(x_1)}+\al\right)u(x_1) >0 
$$
which gives a contradiction to $(\mathcal{L}+\al)u(x_1) \leq 0$.  Therefore, there exists $x_2 \sim x_1$ with  $u(x_1)<u(x_2)$.
Iterating this argument,  we find an infinite path $(x_{n})$ of vertices such that $0<u (x_1)  < u(x_{n})<u(x_{n+1})$ for all $n \in \N$.  From the assumptions on paths  we then obtain
\begin{align*}
\infty &>  \sum_{n=1}^\infty |u(x_{n})|^{p} m(x_{n}) > |u(x_1)|^{p}\sum_{n=1}^\infty m(x_{n})=\infty,
\end{align*}
which contradicts $u \in \ell^p(X,m)$.
\end{proof}

We now turn to geometric criteria for constancy of subharmonic functions. This will need some preparation.  For  $u \colon X\to  \R$ and $x,y\in X$, we define the gradient of $u$ in $x,y$ by
$$\nabla_{x,y} u  = u(x) -u(y).$$
The basic idea in our subsequent considerations is to replace the functions  $u$ in question by the product $u \eta$ for some  $\eta$, which cuts off $u$ in a neighborhood of a ball. Taking the limit for bigger and bigger balls will then give the desired estimates.  In order to carry out this procedure we will need to  control the gradients of $\eta$  and we will have to compare expressions like $\nabla (u \eta)$ or $\mathcal{L}(u\eta) $  with $\eta \nabla u$ and $\eta \mathcal{L} u$.  The  control of the gradient of $\eta$ will stem from our assuming that $\eta$ comes from intrinsic metrics. The comparison of expressions will rely on the subsequent two results.

\begin{lemm}
	Let $(b,c)$ be a graph over $ (X,m) $. Let $u \in \mathcal{F}$ and $\varphi \in \ell^\infty(X)$ be given.
	Then, $u \varphi$ belongs to $ \mathcal{F}$ and for any $x \in X$	
	$$\mathcal{L}(u \varphi)(x) = \varphi(x) \mathcal{L}u(x) + \frac{1}{m(x)} \sum_{y \in X}b(x,y) u(y) \nabla_{x,y}\varphi$$
holds. 	
\end{lemm}
\begin{proof} It is obvious that 
	$ \varphi $ belongs to $\mathcal{F}$. Now, a direct computation gives 
	$$\nabla_{x,y}(u\varphi ) = \varphi(x) \nabla_{x,y}u + u(y) \nabla_{x,y} \varphi.$$
	This easily implies the formula. 
\end{proof}

\begin{propo}[The basic cut-off inequality]\label{basic-cut-off}
	Let $(b,c)$ be a graph over $ (X,m)$. For  $\varphi \in C_c(X)$ and $u
	\in \mathcal{F}$, the inequality 
	$$ \mathcal{Q} (u\varphi )  \leq \sum_{x \in X} \varphi^2(x) \mathcal{L}u(x) u(x) m(x) +
	\frac{1}{2} \sum_{x,y \in X} b(x,y) u^2(x) (\nabla_{x,y}\varphi)^2$$
	holds. 
\end{propo}
\begin{proof} We note that  $u \varphi$ belongs to $ C_c(X)$. Green's formula then gives
	$$\mathcal{Q}(u\varphi)=\sum_{x \in X} \mathcal{L}(u\varphi)(x) (u\varphi)(x)
	m(x).$$ 
	From the preceding lemma we get
	$$\mathcal{Q}(u\varphi)= \sum_{x\in X} (u\varphi)(x) m(x) \left(\varphi(x)\mathcal{L} u(x) + \frac{1}{m(x)} \sum_{y\in X} b(x,y) u(y)  \nabla_{x,y}\varphi\right).$$
	The  sums on the right hand side are absolutely
	convergent as $\varphi$ has finite support. Thus, simple manipulations give 
		$$
	\mathcal{Q}(u\ph)  = \sum_{x \in X} \ph^2(x) \mathcal{L}u(x) u(x)
	m(x) + \frac{1}{2} \sum_{x, y \in X} b(x,y) u(x)u(y)
	(\nabla_{x,y}\ph)^2.
	$$
	Now,
	we can use  the inequality $|u(x)u(y)|\leq \frac{1}{2}(u^2(x)+u^2(y)
	)$ and the  symmetry of $b$ to estimate the second term in the right
	hand side and  obtain the desired statement.
\end{proof}

\begin{thm}[Yau's Liouville theorem]
	Let $(b,c)$ be a connected graph over $ (X,m) $. If there exists an
	intrinsic metric $ \varrho $ such that the restriction of $ \mbox{Deg} $ to any ball with respect to $\varrho$  is bounded,  then every positive
	$ u \in \ell^{p}(X,m) \cap \mathcal F $ for some $ p\in (1,\infty) $ with 
	$$\mathcal{L} u \leq 0$$
	is constant.
\end{thm}
\begin{proof}[Idea of the proof] Here, we discuss   a proof only for a special case. 
  Specifically, we consider $p =2$  and we make the stronger assumption that distance balls with respect to $\varrho$ are finite. {For the proof of the general statement, see \cite[Theorem 12.15]{KLW21}.} 

Fix  $o \in X$ and let  $ B_{s} =B_{s}(o) $ denote the closed ball of radius $s$ around $o$  with respect to the metric $\varrho$. For $ R,r>0 $, we define the  cut-off function $\eta \colon X\to [0,\infty)$ by
	$$
		\eta(x) = \eta_{r,R}(x)=\left(1-\frac{\varrho(x,B_{r})}{R}\right)_{+},
	$$
where $ a_{+}=0\wedge a $ for $ a\in \mathbb{R} $. Then,
$$1_{ B_r}  \leq \eta \leq 1_{B_{R+r}}$$
holds. Moreover,  by definition $\eta$ is a $1/R$-Lipschitz function, i.e., it  satisfies
$$|\eta(x) -\eta(y)|\leq \frac{1}{R} \varrho(x,y).$$
As $\varrho$ is intrinsic, this implies
$$ \frac{1}{2}\sum_{y \in X} b(x,y)(\nabla_{x,y}\eta)^2 \leq\frac{1}{R^2}m(x)$$
for $x \in X$. Then, from 	 $\eta = 1$ on $B_r$,  the basic cut-off inequality in the previous proposition and $\mathcal{L} u\leq 0$ and the preceding estimate  we infer

\begin{multline*} \frac{1}{2} \sum_{x,y\in B_r} b(x,y) (u(x)- u(y))^2  +\sum_{x\in B_r} c(x) u(x)^2 
\leq   Q (u\eta) \\
  \leq  \sum_{x \in X} \eta^2(x) \mathcal{L}u(x) u(x) m(x) +
	\frac{1}{2} \sum_{x,y \in X} b(x,y) u^2(x) (\nabla_{x,y}\eta)^2\\
		\leq  \frac{1}{2} \sum_{x,y \in X} b(x,y) u^2(x) (\nabla_{x,y}\eta)^2
		\leq  \frac{1}{R^2} \|u\|_2^2.
		\end{multline*}
		Taking the limit $R\to \infty$, we infer 
		$$	\frac{1}{2} \sum_{x,y\in B_r} b(x,y) (u(x)- u(y))^2  +\sum_{x\in B_r} c(x) u(x)^2 =0.$$
				As this holds for each $r>0$ and the graph is connected we infer  constancy of $u$. 	
\end{proof}

For a harmonic function $u$, its positive and its negative part are positive subharmonic functions. Thus, we immediately derive the following corollary

\begin{corol}[Yau's Liouville theorem for harmonic functions]
	Let $ (b,c) $ be a connected graph over $ (X,m) $. If there exists 	an intrinsic metric $ \varrho $ such that the restriction of 
	$ \Deg $ to any ball with respect to $\varrho$ is bounded,  then every 	harmonic function   $ u\in\ell^{p}(X,m) $  for some $ p\in
	(1,\infty) $ is constant.
\end{corol}

We finish this section with a Liouville theorem that requires a stronger geometric assumption but then assume less on the subharmonic function.  An  intrinsic metric $\varrho$ is
said to have \textit{finite jump size} if 
$$ \sup\{\varrho(x,y) \mid x,y \in X \text{ with } x\sim y\}<\infty.$$

\begin{thm}[Karp's Liouville theorem]\label{Cacc:t:Karp}\index{theorem!Karp's}\index{theorem!Liouville}
	Let $(b,c)$ be a connected graph over $ (X,m) $. Suppose that there exists an intrinsic metric $ \varrho $ with finite jump size   such that the restriction of 	$ \Deg $ to any ball with respect to $\varrho$ is bounded and let $B_r$ be the ball around a fixed $o\in X$.  Then, a positive $ u \in \mathcal F$ with $\mathcal L u \leq 0$  is constant if
	\begin{align*}
		\int_{r_{0}}^{\infty}\frac{r}{\| u1_{B_{r} }\|_{p}^{p}}dr =\infty
	\end{align*}
	for some $p\in (1,\infty)$ and  some  $r_0\geq 0$ with $u 1_{B_{r_0}} \neq 0$.
\end{thm}

\begin{rem} Yau's theorem was  first proven by Yau for manifolds in \cite{Yau76}. It is a  corollary of Karp's theorem shown in \cite{Karp82}.  This was later extended to strongly local Dirichlet forms by Sturm in \cite{Stu94}. For graphs, the general statements given here are taken from  Hua/Keller \cite{HK14} improving upon earlier  results from \cite{HS97, RSV97, Mas09, HJ14}. Recently, an extension to general regular Dirichlet forms was established in \cite{HKLS22}.
\end{rem}

\section{Stochastic completeness or the case $p =\infty$} \label{sec-Grigoryan}
In this section we study subharmonic functions in $\ell^p$ for $p =\infty$. 
This is strongly linked to a property of the heat equation known a stochastic completeness.
Specifically, a graph  $(b,0)$ over $(X,m)$ is said to be \textit{stochastically complete} if for any $f\geq 0$ with $f\in \ell^1 (X,m)\cap \ell^2 (X,m)$ the quantity   
$$\sum_{x\in X} e^{- t L } f (x) m(x) $$
does not depend on $t\geq 0$. If one thinks of $e^{-t L} f$ as the distribution of heat on $X$ at time $t$ for an initial distribution given by $f$, then stochastic completeness just means that the total amount of heat is conserved (and this explains the name). It turns out that one can extend the operators $e^{-t L}$  from $\ell^2 (X,m)$ to $\ell^\infty (X)$ and with this extension denoted by $e^{-t L}$ again we can formulate stochastic completeness briefly as
$$e^{-tL} 1 = 1$$
for all $t\geq 0$ {which relates to the definition of stochastic completeness given above via the calculation
\begin{multline*}
\sum_{x\in X} e^{- t L } f (x) m(x) = \sum_{x\in X} e^{- t L } f(x) 1 (x) m(x) =\langle e^{-tL}f,1\rangle_{1,\infty}\\= \langle f, e^{-tL}1\rangle_{1,\infty}
=\langle f,1\rangle_{1,\infty}=\sum_{x\in X} f(x) m(x),
\end{multline*}
where $ \langle\cdot,\cdot\rangle_{1,\infty} $ denotes the dual pairing between $ \ell^1(X,m) $ and $ \ell^\infty(X) $ and $ e^{-tL}  $ on $ \ell^{1}$ is the adjoint of $ e^{-tL} $ on $ \ell^{\infty} $.
} Now, the connection to Liouville theorems comes from the following.

\begin{thm}[Stochastic completeness and Liouville property in $\ell^\infty$] Let $b$ be a graph over $(X,m)${, i.e., $ c=0 $}. Then, the following assertions are equivalent:
\begin{enumerate}[(i)]
\item  The graph is stochastically complete.
\item  For any $f\in \ell^\infty (X)$ there exists a unique bounded solution $u$ of the heat equation
$$(\mathcal{L} + \partial_t) u =0$$
with initial condition $f$. 
\item  Any nonnegative $u\in \ell^\infty (X)$ with $(\mathcal{L} +\alpha) u\leq 0$ for some $\alpha >0$ is trivial.
\end{enumerate}
\end{thm}

 \begin{rem} There is a long history to the  study of stochastic completeness both in the continuous setting, see, e.g., the survey of Grigor'yan \cite{Gri99}, and the discrete setting, see, e.g.,  the early work of Feller \cite{Fel57} and Reuter \cite{Reu57}. Often the term \textit{conservativness} is used instead of stochastic completeness.  In the specific case of graphs with standard weights and counting measure, the previous theorem  (and more) is worked out in \cite{Woj08}. In the generality stated here the result is due to  \cite{KL}, where also an extension called ``stochastic completeness at infinity'' is discussed to deal with graphs with non-vanishing $c$. 
 \end{rem}

In order to be able to deal with distance balls of small measure we define 
$ \log^{\#} (t) :=\log (t) \vee 1$ for $t>0$.

\begin{thm}[Grigor'yan theorem]
Let $ b $ be a graph over $ (X,m) $ and let $ \varrho $ be an intrinsic metric all of whose distance balls are finite. If
	\begin{align*}
		\int_{0}^{\infty}\frac{r}{\log^{\#}( m(B_{r}))}dr=\infty,
	\end{align*}
	then the graph is stochastically complete.
\end{thm}

\begin{rem} The results of this section are analogues to   results on  stochastic completeness of Riemannian manifolds going back to work of Gaffney \cite{Gaf59}, Karp/Li \cite{KL} and culminating in in Grigor'yan's volume growth result in \cite{Gri86}, see also the survey
\cite{Gri99}.  Other approaches to volume growth criteria for manifolds can be found in \cite{Dav92, Tak89}. An  extension of Grigor'yan's result to strongly local Dirichlet forms is established in \cite{Stu94}.

For graphs, the volume growth criterion of Grigor'yan fails badly when one uses the standard graph metric, see \cite{Woj08, Woj11}.  This  spurred the use of  intrinsic type  metrics to the graph
setting to find  analogues for the result of Grigo'yan as done first by  
Grigor'yan/Huang/Masamune for jump processes in \cite{GHM12}; however, their result did not yield the optimal
volume growth bound on graphs.  
This optimal result was  achieved by Folz \cite{Fol14}  who used probabilistic techniques to compare the discrete graph to a continuum object (quantum graph) and then appealed to  Sturm's extension of Grigor'yan's result.  
Huang \cite{Hua14b} achieved similar results using quantum graphs
and analytic techniques while Huang/Shiozawa \cite{HS14} gave a probabilistic proof using refinements instead of quantum graphs.   Our result here is taken from  Huang/Keller/Schmidt \cite{HKS20}  which gives a purely analytic proof of Grigor'yan's volume growth  and removes all restrictions on the graph structure found in the previous results on graphs in \cite{Fol14, Hua14b, HS14}.
\end{rem}

\part{Recurrence and (metric) boundaries}
Recurrence is a core concept and can be  phrased in many different ways. A prominent way  involving the Laplacian is by absence of superharmonic functions.  The word  recurrence  itself comes from a formulation via  Markov processes. Markov processes model the behavior of particles moving on a  given set according to some basic assumptions. These assumptions make sure that the future movement of a particle does not depend on its past movement but only on its present position.  Recurrence then
describes the phenomenon  that the particle comes back 
again and again. In the geometric-analytic description, which is our
concern here, this coming back  is encoded by various forms of  irrelevance of
what is happening far away. So, recurrence is very much a property determined at  infinity in the sense that infinity does not play a role for it.  

Below we first define recurrence in a way that makes this irrelevance of infinity apparent (Section~\ref{sec-basics-recurrence}). We then turn to a quantitative way of phrasing this. There, the concept of infinity is made precise via boundaries of dense embeddings of the graph and the concept of irrelevance is made precise by vanishing capacity (Section \ref{sec-boundary-capacity}).  Finally, we then present our main result (Section \ref{sec-main-recurrence}).  The material  in Section~\ref{sec-basics-recurrence} is standard, see e.g. \cite{KLW21} for further discussion, references and history {or \cite{Woe} for a more probabilistic point of view}. The material in the subsequent sections of this part is taken from  the article \cite{LPS23} if not stated otherwise.

\section{The definition of recurrence}\label{sec-basics-recurrence}
In this section we will deal with connected graphs only. Indeed, from our interpretation of the particle coming back again and again it makes sense to make this assumption.

\begin{defi}[Recurrence]
A connected   graph $(b,c)$ over $X$ is called \emph{recurrent} if there  exists a sequence $(\varphi_n)$ in $C_c (X)$ 
 with $\varphi_n \to 1$
pointwise and $\mathcal{Q}(\varphi_n)  \to 0$, as  $n\to \infty$.
\end{defi}

By lower semicontinuity of  $\mathcal{Q}$ recurrence of the graph implies  $\mathcal{Q}(1) =0$ and this gives $c =0$. This is well in line with the idea that the particle comes back again and again, which is  incompatible with the presence of a killing term $c$, which removes the random walker from the graph with a certain probability. So from now on we will only consider graphs with $c =0$ in the context of recurrence. 

 We note that the definition of recurrence  can then be rephrased as existence of {a so called \emph{null sequence}, which is }a sequence $(\varphi_n)$ in $C_c (X)$ converging to $1$ pointwise and with respect to $\mathcal{Q}$, with the latter meaning
$$\mathcal{Q}(1-\varphi_n)  =\mathcal{Q}(\varphi_n) \to 0,\quad n\to \infty.$$
  This can be seen  as an instance of how the
behavior outside of compact sets (in this case the supports of the
$\varphi_n$) becomes irrelevant. Intuitively speaking one may say that in a recurrent graph there is no energy sitting at infinity. 
 A convenient functional analytic  way to phrase recurrence involves the 
the (pseudo-)norm $\aVd_o \colon \mathcal{D} \to [0,\infty)$  defined by
 $$\aV{f}_o := \sqrt{\mathcal{Q}(f) + |f(o)|^2},$$
for $o\in X$ arbitrary. 
 If the graph is connected, then $\aVd_{o}$ is  a norm and 
$\aVd_{o}$ and $\aVd_{o'}$ are equivalent for $o,o'\in X$. Moreover,   $(\mathcal{D},\aVd_{o})$ is a Banach space  and  $f_n \to f$ with respect to $\aVd_{o}$ implies $f_n \to f$ pointwise.
 We then denote the  closure of the space of functions of compact support $C_c(X)$ with respect to $\aVd_o$  by $$ \mathcal{D}_0 = (\mathcal D_0)_{b,c}. $$
Then $\mathcal{D}_0$ is a Banach space (as it is a closed subspace of a Banach space) and it does not depend on $o\in X$ (as the norms $\aVd_{o}$ are all equivalent. 

\begin{thm}[Characterizations of recurrence] For a connected graph $b$ over $X$ the following statements are equivalent:
\begin{enumerate}[(i)]
\item  The graph is recurrent.
\item   The constant function $1$ belongs to $\mathcal{D}_0$.
\item  $\mathcal{D} = \mathcal{D}_0$.
\end{enumerate}
\end{thm}

\begin{rem} Connected graphs that are not recurrent are called \textit{transient}. 
For disconnected graphs transience is a stronger property than not being recurrent. In this section we restrict attention to connected graphs.  For further background on recurrence, we refer the reader to \cite{Schmi}.
\end{rem}

\section{Dense embeddings and  the capacity of their boundary}\label{sec-boundary-capacity}
In this section we discuss a concept of boundary and a concept of capacity in our setting. 
As outlined in the introduction to this part, we will deal with recurrence by  encoding infinity by  suitable boundaries and  its  irrelevance by vanishing capacity.

Let $X$ be a countable set. Let $Y$ be a topological Hausdorff space.
 We say that $X$
\textit{embeds densely} in the topological space $Y$ if $Y$
contains a copy of $X$, the restriction of the topology of $Y$ on
$X$ is the discrete topology, and $X$ is dense in $Y$.  Clearly, $Y$ must be separable whenever $X$ embeds
densely in it. Whenever $X$ embeds densely in $Y$, we define the
\textit{boundary} $\partial_Y X$  of $X$ in $Y$ by
$$\partial_Y X:= Y\setminus X.$$
The complement (in $Y$)  of any finite subset of $X$ is
open in $Y$ (as any finite set is compact  and then must be closed
due to Hausdorff property). Hence, any  such a complement is an open
neighborhood of $\partial_Y X$. In particular, any function $h$ with
finite support on $X$ can be extended (by zero) to a continuous
function on $Y$.

If $X$ embeds densely in a  compact $Y$, then $Y$ is called a
\textit{compactification} of $X$. In this case the  open
neighborhoods of $\partial_Y X$ are exactly given by the complements
of finite sets of $X$. A particular instance  is given by the
\textit{one-point-compactification}. It is given by the set $Y=X\cup
\{{\infty}\}$, where ${\infty}$ is an arbitrary additional point,
and this set is equipped  with topology given by the family of all
subsets of $Y$ that are either subsets of $X$ or whose complement is
finite. In this case the boundary of $\partial_Y X$ is just
${\infty}$.

A particular instance of dense embeddings comes about from metrics on $X$.  This is discussed next. 
Let $\sigma$ be a pseudometric on $X$.
The completion of $X$ with respect to $\sigma$ is defined as the set
of equivalence classes of $\sigma$-Cauchy sequences in $X$, where
two such sequences $(x_n)$ and $(y_n)$ are considered to be
equivalent if
$$ \lim_{n\rightarrow\infty}\sigma(x_n, y_n) = 0. $$
 This set is denoted by $\overline{X}^\sigma$ and contains a
quotient of the vertex set $X$ as the classes of the constant
sequences.  Clearly, $\sigma$ can be extended to a pseudometric
on $X$ and this extension will  -- by a slight abuse of notation --
also denoted by $\sigma$. Subsequently, the boundary with respect to $\sigma$ is defined as
$$\partial_\sigma X = \overline{X}^\sigma \setminus (X/\simeq),$$
where $x\simeq y$ if $\sigma(x,y) = 0$.  A graph is called \emph{metrically complete} with respect to a pseudometric if the
boundary is empty. Clearly, if $\sigma$ is a metric,  then $\overline{X}^\sigma$ contains a copy of $X$, this copy is dense, and
our  definition of metric completeness  agrees with the usual definition (that any Cauchy-sequence
converges).   So, in this case  $X$ embeds densely in  $\overline{X}^\sigma$ and  
$$\partial_\sigma X = \partial_{\overline{X}^\sigma} X$$
holds.

Let $b$ be a graph over $ X $ and let $m \colon X \to (0,\infty)$ be a
measure. The \emph{capacity} of a subset $U\subset X$ is defined
by
$${\rm cap}_m(U) := \inf \{\mathcal Q(f) + \av{f}_{\ell^2(X,m)}^2  \mid f\in \mathcal D \cap \ell^2(X,m)  \text{ with } f\geq  1 \text{ on } U\},$$  
with the convention that ${\rm cap}_m(U) = \infty$ if the set in the above definition is empty. {In fact,   this is just the capacity of $Q^{(N)}$ with respect to the measure $m$, viz
 $$ {\rm cap}_m(U) = \inf\{Q_{1}^{(N)}(f)\mid f \in D(Q^{(N)})  \mbox{ with }  f\ge 1 \mbox{ on } U \}.$$}
 We  can assume $0\leq f\leq 1$ in this definition (as otherwise we could replace $f$ by  its contracted version $(f \wedge 1)_+=0\vee f\wedge 1$ that satisfies the same constraints  as $f$ but gives a reduced value in the infimum.

Whenever $X$ embeds densely in  $Y$, we can extend the capacity
to subsets of $Y$ by setting
$$ {\rm cap}_m(A) := \inf \{ {\rm cap}_m(O\cap X) \mid  O \text{ open in } Y \text{ with }  A \subset O\}. $$
%
%
%
%
%
On $X$ this definition agrees with the earlier defined capacity as 
 by assumption every subset $A \subset X$ is open  in $Y$ (as
 the topology of $Y$ induces the discrete topology on $X$). 
 The capacity is an outer measure on the power set of $Y$ with $m(A)
\leq {\rm cap}_m(A)$ for all $A \subset X$ and ${\rm cap}_m(Y)
\leq m(X)$, see e.g.   \cite[Theorem ~2.1.1 and Theorem~A.1.2]{FOT}.

A main   insight for our considerations is that  vanishing of capacity at the boundary is related to infinite distance of the boundary.  More specifically, the following lemma holds {which is taken from \cite[Corollary~3.4]{LPS23}}.

\begin{lemm}[Capacity zero sets in the boundary  have infinite distance]\label{cor-infinite-distance}  Assume that $X$ embeds densely into $Y$ and that $b$ is a graph over $X$. Then, the following assertions for $A\subset \partial_Y X$ are equivalent:

\begin{enumerate}[(i)]
\item  $ A$ has zero capacity (with respect to any finite measure).
\item There   exists an intrinsic metric   $\varrho$ with respect to a finite measure such that for any finite $F\subset X$ and any  $r>0$ there exists an open neighborhood $U$ of $A$ in $Y$  with
$$\varrho(U\cap X,F):= \inf \{ \varrho(z,x) \mid z\in U\cap X, x\in F\} \geq r.$$
\end{enumerate}
If (i) and (ii) hold, the intrinsic metric $\varrho$ can even  be chosen as path metric. 
\end{lemm}
The basic  \textit{idea of the proof} is that vanishing of the capacity means the existence of functions $(f_n)$ of smaller and smaller energy each of constant value $1$ on $A$. Now, suitable summation over these $(f_n)$ gives a function $f$ of finite energy,  which tends to infinity on $A$. Then, $\varrho (x,y) = |f(x) - f(y)|$ is the desired metric.  Conversely, the existence of a metric as in (ii) allows  one to consider  for a given  finite $F$ and $r =1$ an  $U$ as in (ii) and  the function  $f_U := (1-\varrho (U,\cdot))_+$.  Then  $f_U$ is a $1$-Lipschitz function  with value $1$ on $A$, vanishing on $F$  and an estimate of Section \ref{sec-intrinsic} gives
$$\mathcal{Q}(f_U) + m(f^2) \leq  (C^2 +1) m(X\setminus F).$$
As $m$ is finite, the right hand side can be made arbitrarily small by choosing large enough $F$.


\section{Characterizations of recurrence via boundaries}\label{sec-main-recurrence}
The main abstract result of this section provides various characterizations of recurrence. This is then used to obtain a sufficient metric criterion for recurrence. Not surprisingly this criterion deals with vanishing capacity of the boundary, {see \cite[Theorem~4.2]{LPS23}}.

\begin{thm}[Characterization of recurrence] \label{recurrence}
Let $b$  be an infinite graph over $X$. The following conditions are
equivalent:

\begin{enumerate}[(i)]
 \item  $b$ is recurrent.

 \item   There exists a function of finite energy $ f\in \mathcal{D} $  that satisfies
 $$\liminf_{x\rightarrow\infty} f(x) = \sup_{F \subset X \text{ finite}} \inf \{f(x)  \mid x \in X \setminus F \} = \infty.$$

 \item There is an intrinsic metric $\sigma$ with respect to a finite measure that
induces the discrete topology on $X$ such that distance balls with
respect to $\sigma$ are finite.

\item[(iii')] There exists a finite measure $m$ and an edge weight
$w$ adapted to it such that the distance balls with respect to $d_w$
are finite.

\item For one  (every) finite
measure $m$ on $X$ and  one (every) compactification $Y$ of $X$ the
equality ${\rm cap}_m (\partial_Y X) = 0$ holds.

\item One (every) finite measure has the following feature: For any
$\varepsilon>0$ there exists a finite set $F\subset X$ with ${\rm
cap}_m (X\setminus F) <\varepsilon$.
\end{enumerate}
%

\end{thm}
We provide a \textit{sketch of the proof}: The equivalence between (i) and (v) can be understood as follows:  Condition (i) is about  existence of  a sequence $(\varphi_n)$ in $C_c (X)$  converging to $1$ pointwise with $\mathcal{Q}(\varphi_n)\to 0$. Condition (v) on the other hand is a about a sequence $(\psi_n)$ with $\psi_n =1$ outside of  larger and larger compact sets with $\mathcal{Q}(\psi_n) \to 0$.  Now, the equivalence follows by taking $\psi_n = 1-\varphi_n$.  The equivalence between (v) and (iv) follows as the boundary of compactifications is the intersection of the complements of finite sets.  The implication (iv) $\Rightarrow$ (iii') is essentially contained in Lemma \ref{cor-infinite-distance}.  The implication (iii') $\Rightarrow$ (iii) is clear.  To derive (ii) from (iii), we consider an intrinsic metric $\sigma$ as in (ii) and set $f(x):= \sigma(p,x)$ for a fixed $p\in X$. Then,  $f$ is a $1$-Lipschitz function with respect to $\sigma$ and, hence, has finite energy by what we have discussed in the section on intrinsic metrics. Moreover, it has the desired limiting behavior  by finiteness of distance balls. It remains to discuss the implication (ii) $\Rightarrow$ (i): Here, the basic idea is to use functions of the form $\psi_n = f\wedge (n+1) - f\wedge n$ to bound the capacity.  This finishes the discussion of the proof.

The theorem gives a metric handle on recurrence. This is discussed in the remaining part of this section.  The first result 
is an immediate consequence of (part (v) of) the theorem and provides a necessary metric condition for recurrence.

\begin{corol}[Metric capacity criterion - necessary condition] \label{coro:polarity boundary}
Let $b$ be a recurrent infinite graph over $ X $. For any finite
measure $m$ on $X$ and any metric $\sigma$ on $X$ that induces the
discrete topology, we have
 $${\rm cap}_m(\partial_\sigma X)~=~0.$$
\end{corol}

To prove a converse to the corollary and, hence, obtain a sufficient metric criterion for recurrence, we need some additional  compactness.  One instance of such compactness is valid for locally finite graphs and comes from the following Hopf-Rinow type theorem.

\begin{thm}[Hopf-Rinow type theorem \cite{HKMW}, \cite{KM19}) as well] \label{Keller}
Let $b$ be a locally finite graph over $ X $ and let $w$ be an edge weight. Then $d_w$ is a metric that induces the discrete topology on
$X$. Moreover, the following assertions are equivalent:
\begin{enumerate}[(i)]
 \item $(X,d_w)$ is a complete metric space.
 \item $(X,d_w)$ is geodesically complete, i.e., every infinite path has infinite length with respect to $w$.
 \item Every distance ball is finite.
 \item Every bounded and closed set is compact.
\end{enumerate}
\end{thm}

From the compactness feature provided by the Hopf-Rinow theorem and the abstract characterization we then obtain by a short argument  the following sufficient metric condition for recurrence.

\begin{corol}[Metric capacity criterion -  sufficient condition for locally finite graphs] \label{cap-recurrence}
A locally finite  $b$ over $X$  is recurrent if (and only if)  
$${\rm cap}_m(\partial_\sigma X) = 0$$
holds for some metric $\sigma$  that is intrinsic metric with respect to some finite measure  and induces the discrete topology.
\end{corol}

If $\overline{X}^\sigma$ is compact, then we can drop the local finiteness assumption. 

\begin{corol}[Metric capacity criterion - sufficient condition involving compactness]
A graph $G = (X,b)$ is recurrent if 
$${\rm cap}_m(\partial_\sigma X) = 0$$
holds for some metric  $\sigma$ that is   intrinsic metric with respect to a finite measure, induces the discrete topology   and has the property that $(X,\sigma)$ is totally bounded.
\end{corol}

\part{Characterization of Dirichlet forms  via the boundary}
In the context of the Laplacian $\Delta$ on an open subset $\Omega \subset \R^d$ it is natural to look for self-adjoint realizations of $\Delta$ on $L^2(\Omega)$. More precisely, one looks for operator domains $D \subset L^2(\Omega)$, for which $\Delta|_D$ is self-adjoint. Such self-adjoint realizations can be encoded by boundary conditions. In our context of Dirichlet forms the analogue quest is to look for  Dirichlet forms which are restrictions of a given energy form $\mathcal{Q}$. Somewhat more generally, we may  even  consider forms $Q$  that lie between
$Q^{(N)}$ and $ Q^{(D)}$ in a suitable sense. 
It turns out that such Dirichlet forms can also be described via boundaries. Specifically,  with a Dirichlet type form $q'$ on the boundary, we find 
$$Q = Q^{(N)} + q'$$
as well as a variant with $Q^{(D)}$ instead of $Q^{(N)}$. 
In some cases we even get a description of $Q$  via boundary values of the functions in questions. 
All of this, of course, requires a boundary in the first place. Providing such a boundary is the first task in this part (Section \ref{sec-boundary}). With such a boundary at hand we can then derive formulae as a above for $Q$ and even describe the forms by boundary values of the foundations (Section \ref{sec-Q}). We can then even go on and study the set of all Dirichlet forms on $(X,m)$ (Section \ref{ref-allDF}).

{The results in this part are based on \cite{KLSS19,KLSSW23}. As mentioned in the introduction,  we provide rather complete proofs in this part as  both references treat slightly different boundaries and we consider here a somewhat more specialized situation. }

\section{The harmonic boundary, traces, extensions and measures}\label{sec-boundary}
In this section we introduce a topological boundary that can describe all $1$-harmonic functions of finite energy for a given graph. It goes back to the work of Kayano and Yamasaki \cite{KY88}, see also \cite[Chapter~VI]{Soa}.  These functions play a key role for classifying all Markovian realizations of the discrete Laplacian (the self-adjoint realizations induced by Dirichlet forms).   Throughout this section we fix a graph $(b,c)$ over the discrete measure space $(X,m)$.

 \begin{rem}[$\ell^2$-theory v.s. measure free theory] \label{remark: ell two vs measure free}  We mainly {build on} results that were obtained in a slightly different setting for transient graphs. More precisely, we will transfer results for harmonic functions of finite energy to $1$-harmonic functions of finite energy that are also in $\ell^2(X,m)$. There are three observations making this possible:
  \begin{enumerate}[(a)]
   \item $\cL_{b,c} + 1 = \cL_{b,c + m}$. In particular,  $1$-harmonic functions for the graph $(b,c)$ correspond to harmonic functions for the graph $(b, c + m)$.
   \item $$D(Q^{(N)}) = \cD_{b,c} \cap \ell^2(X,m) = \cD_{b, c + m}$$ and $$D(Q^{(D)}) = (\cD_{0})_{b,c} \cap \ell^2(X,m) = (\cD_0)_{b,c+m}.$$
   \item The graph $(b,c + m)$ is transient because $c(x) + m(x)  > 0$ for all $x \in X$.  
  \end{enumerate}
 \end{rem}

\subsection{The harmonic boundary and the trace map} \label{section:harmonic boundary}
 In this section we discuss the actual boundary in question and how to restrict the functions in question to this boundary.

\begin{propo}[Existence of a compatible compactification]
There exists a unique (up to homeomorphism) compact Hausdorff space $\hat X$ with the following properties:
\begin{enumerate}[(a)]
 \item $X$ is a dense subset of $\hat X$ and for any $x \in X$ the singleton set $\{x\}$ is open in $\hat X$. 
 \item Any $f \in  D(Q^{(N)}) \cap \ell^\infty(X)$ can be extended to a continuous function $\hat f \colon \hat X \to \R$.
  \item $\{\hat f \mid f\in  D(Q^{(N)}) \cap \ell^\infty(X)\}$ separates the points of $\hat X$. 
\end{enumerate}
\end{propo}
\begin{proof} This is a standard application of Gelfand theory using the fact that $D(Q^{(N)}) \cap \ell^\infty(X)$ (or rather its complexification) is an algebra. For details, we refer to  \cite[Theorem~6.4]{Soa} (which considers $\cD_{b,0} \cap \ell^\infty(X) $ instead of  $D(Q^{(N)}) \cap \ell^\infty(X)$)  or \cite[Section~4]{HKLW} (which only treats  the case when  $D(Q^{(N)}) = \mathcal D_{b,c+m} \subset \ell^\infty(X)$). The arguments given in both references can be easily adapted to treat $D(Q^{(N)}) \cap \ell^\infty(X)$.
%
\end{proof}

\begin{rem}
 The space $\hat X$ is the {\em Royden compactification} of $X$ with respect to the graph $(b,c+m)$.
\end{rem}

Unless necessary we drop the superscript hat in the notation of the extended function, i.e., we tacitly extend any function in $f\in D(Q^{(N)}) \cap \ell^\infty(X)$ to a continuous function on $\hat X$ denoted by $f$ again.   The set 
$$\partial X = \{x \in \hat X \setminus X \mid  f(x) =  0 \text{ for all } f \in D(Q^{(D)}) \cap \ell^\infty(X)\}$$
is called {\em $1$-harmonic boundary}.  As it is an intersection of closed subsets of the compact set $\hat X$ it is compact. 
The  map 
$$\gamma_0 \colon D(Q^{(N)}) \cap \ell^\infty(X) \to C(\partial X), \quad f \mapsto  f|_{\partial X}$$
is called {\em preliminary trace map}. Clearly, it is an algebra homomorphism  and a lattice homomorphism (with respect to taking minimum and maximum of $\R$-valued functions) with
$$D(Q^{(D)}) \cap \ell^\infty(X) \subset \ker \gamma_0.  $$

\begin{rem}[Why we consider $\partial X$ and not $\widehat{X}\setminus X$]
Below we will see that $D(Q^{(D)}) \cap \ell^\infty(X) = \ker \gamma_0$. This formula is the reason why we only consider $\partial X$ as the boundary of $X$ and not the entire set $\hat X \setminus X$. The equality $\partial X = \hat X \setminus X$  is equivalent to uniform transience of the graph $(b,c + m)$, see \cite[Theorem~4.2]{KLSW17}. 
\end{rem}

We denote by $\overline \R = [-\infty,\infty]$ the two-point compactification of $\R$.

\begin{propo}[Extending elements of $D(Q^{(N)})$ to $\widehat{X}$.]\label{proposition:unbounded extension}
 Any $f \in D(Q^{(N)})$ can be uniquely extended to a continuous function $\hat f \colon \hat X \to \overline \R$ and if $f \in D(Q^{(D)})$, then $\hat f = 0$ on $\partial X$.
 
 More precisely, for $f \in D(Q^{(N)})$ and $n \in \N$, we have $(f \wedge n) \vee (-n) \in D(Q^{(N)}) \cap \ell^\infty(X)$ and  
 $\hat f =   \lim_{n \to \infty} (f \wedge n) \vee (-n)$
 pointwise on $\hat X$ (where as before $(f \wedge n) \vee (-n)$ is tacitly extended to $\hat X$). 
\end{propo}

\begin{rem}
  Before giving the proof, there is a little warning in place here. For $x \in \hat X \setminus X$ the quantity $(f(x) \wedge n) \vee (-n)$ is not (yet)  defined, while $((f \wedge n) \vee (-n))(x)$ is defined via the unique continuous extension of $(f \wedge n) \vee (-n) \in D(Q^{(N)}) \cap \ell^\infty(X)$ to $\hat X$.
\end{rem}

\begin{proof}
We basically follow the proof of \cite[Theorem~6.9]{Soa}. The density of $X$ in $\hat X$ implies that there is at most one continuous extension of $f$, hence it suffices to show its existence via the formula as pointwise limit. 

First suppose that $f \geq 0$. Clearly, $f \wedge n \in D(Q^{(N)})\cap \ell^\infty (X)$ with $f\wedge n\leq f\wedge m$ for $n\leq m$. Hence,  the limit $\hat f := \lim_{n \to \infty} f \wedge n$ exists pointwise and $\hat f$ is an extension of $f$ to $\widehat{X}$.
  To prove its continuity, we distinguish two cases: 

\textit{Case 1: $x \in \hat X$ with $\hat f(x) < \infty$.}  We choose $n > \hat f(x)$, which implies $(f \wedge n)(x) \leq \hat f(x)  < n$. By the continuity of $f \wedge n$, we find an open neighborhood $U$ of $x$ in $\hat X$ such that $f \wedge n < n$ on $U$. As extending functions to $\hat X$ is compatible with the lattice operations, we infer 
$$(f\wedge m) (y) \wedge n = (f\wedge n) (y)< n$$
for all $n\leq m $ and $y\in U$. This gives $(f\wedge m) (y) = (f\wedge n) (y)$ for all $y\in U$ and $n\leq m$. From this, we infer $\hat f  = f\wedge n$ on $U$ and continuity in $x$ follows. 

\textit{ Case 2: $\hat f(x) = \infty$.} By definition, for each $C > 0$, we find $n \in \N$ such that $(f \wedge n)(x) > C$. The continuity of $f \wedge n$ yields an open neighborhood $U$ of $x$ in $\hat X$ such that 
 $f \wedge n > C$ on $U$. Since $\hat f \geq f \wedge n$, this implies $\hat f > C$ on $U$ and shows the continuity of $\hat f$ in $x$.

 So far, we only treated nonnegative $f$. For arbitrary $f \in D(Q^{(N)})$, we have 
 $$(f \wedge n) \vee (-n) = (f_+ \wedge n) - (f_- \wedge n).$$
 We already know that the summands on the right side converge to continuous functions $\widehat{f_{\pm}} \colon \hat X \to [0,\infty]$, which extend $f_\pm$. In order to show that $\widehat{f_{+}} - \widehat{f_{-}}$ exists and is continuous,
it suffices to show that for all $x \in X$ with $\widehat{f_{+}}(x) = \infty$ there exists an open neighborhood $U$ of $x$ such that $\widehat{f_{-}} = 0$ on $U$.  If $\widehat{f_{+}}(x) = \infty$, then there exists $n \in \N$ such that $(f_+ \wedge n) (x) > 0$. The continuity of $f_+ \wedge n$ yields an open neighborhood $U$ of $x$ such that $f_+ \wedge n > 0$ on $U$. For all $m \geq n$ this implies $f_- \wedge m = 0$  on $U \cap X$. Using the continuity of $f_- \wedge m$ and the density of $X$ in $\hat X$,  we infer $f_- \wedge m = 0$ on $U$, $m \geq n$, and arrive at $\widehat{f_{-}} = 0$ on $U$.

 For the statement on $f \in D(Q^{(D)})$, we simply note that in this case $(f \wedge n) \vee (-n) \in D(Q^{(D)}) \cap \ell^\infty(X)$ since $ Q^{(D)} $ is a Dirichlet form. By the definition of $\partial X$ this implies  $(f \wedge n) \vee (-n) = 0$ on $\partial X$. 
\end{proof}

As before,  when dealing with bounded functions in $D(Q^{(N)})$, we drop the superscript hat on $f$  and tacitly extend any $f \in D(Q^{(N)})$ to a continuous function on $\hat X$. The map 
$$\gamma \colon D(Q^{(N)}) \to C(\partial X,\overline \R), \quad f \mapsto  f|_{\partial X}$$
is called {\em trace map}. It extends $\gamma_0$ and  the  previous proposition shows 
$$\gamma (f) = \lim_{n \to \infty} \gamma_0 ((f \wedge n) \vee (-n)),$$
with the limit existing pointwise. 

We next discuss how the trace map is compatible with various operations: The set $C(\partial X,\overline \R)$ is not a vector space but a lattice with respect to taking pointwise maxima and minima. With respect to these operations the trace map is a lattice homomorphism. If $f \in D(Q^{(N)})$ with $\gamma (f) \in C(\partial X,\R)$, then for all $g \in D(Q^{(N)})$ the sum $\gamma(f) + \gamma(g)$ is well-defined in $C(\partial X,\overline \R)$ and we have  $\gamma(f + g) = \gamma(f) + \gamma(g)$. In this sense, $\gamma$ is linear.

If $C \colon \R \to \R$ is a monotone normal contraction, then the limits $C(\pm \infty) := \lim_{t \to \infty} C(\pm t)$ exist in $\overline \R$. With this extension, we have
$$C(\gamma(f)) = \gamma(C(f)), \quad f \in D(Q^{(N)}).$$
In this sense, the trace map is compatible with monotone normal contractions. 

Below we shall see that for $f \in D(Q^{(N)})$ the function $\gamma(f)$ only attains the ``problematic'' values $\pm \infty$ on a very small set.


\subsection{Royden decomposition, maximum principle and measures} \label{section:Royden}
In the last section we have discussed how a function $f\in D(Q^{(N)})$ can be extended to $\hat X$. Here, we first present a decomposition of any $f$ into two parts, one vanishing on the boundary, and the other $1$-harmonic. We then discuss how $1$-harmonic $f$ are determined by their boundary value and derive various consequences.

\begin{propo}[Royden decomposition]\label{proposition:royden}

\begin{enumerate}[(a)]
 \item To any $f \in D(Q^{(N)})$ there exist unique $f_h \in D(Q^{(N)})$ and $f_0 \in D(Q^{(D)})$ with $f = f_h + f_0$ and 
$\cL f_h + f_h = 0$. Moreover,   $$Q^{(N)}_1(f) = Q^{(N)}_1(f_0) + Q^{(N)}_1(f_h)$$
holds  
 and the map  $(D(Q^{(N)}),\av{\cdot}_{Q^{(N)}}) \to (D(Q^{(N)}),\av{\cdot}_{Q^{(N)}})$,  $f \mapsto f_h$ is continuous.
 \item If for $-\infty \leq \alpha \leq 0 \leq \beta \leq \infty$ the function $f\in D(Q^{(N)})$ satisfies $\alpha \leq f \leq \beta
  $, then $\alpha \leq f_h \leq \beta$.
\end{enumerate}
\end{propo}
\begin{proof}
This is the Royden decomposition discussed in  \cite[Proposition~5.1]{KLSW17} applied to the transient graph $(b,c + m)$, noting that harmonic functions with respect to  $(b,c + m)$ correspond to $1$-harmonic functions with respect to $(b,c)$, see  Remark~\ref{remark: ell two vs measure free}.
\end{proof}

The following maximum principle is implicitly contained in \cite[Proposition~5.2]{KLSW17}. Since it may be of use in other contexts as well, we include a proof.

\begin{thm}[Maximum principle]
Assume that $f \in D(Q^{(N)})$ satisfies $\cL f + f \geq 0$ on $X$ and $f \geq 0$ on $\partial X$. Then, $f \geq 0$. 
\end{thm}

In order to establish this result, we need two auxiliary lemmas. One is   the maximum principle for functions in $D(Q^{(D)})$, which according to Proposition~\ref{proposition:unbounded extension} vanish on $\partial X$.

\begin{lemm}\label{lemma:positivity of superharmonic functions}
 If $f \in D(Q^{(D)})$ satisfies $\cL f + f \geq 0$, then $f \geq 0$. 
\end{lemm}
\begin{proof}
Since $Q^{(D)}$ is a Dirichlet form, we have $f_- \in D(Q^{(D)})$. Hence, we find a sequence of nonnegative function $(g_n)$ in $C_c(X)$ with $g_n \to f_-$ with respect to $Q^{(D)}_1$. Using   $\cL f + f \geq 0$, $Q^{(D)}_1(f_+,f_-) \leq 0$ and Green's formula, we infer
\begin{multline*}
-Q^{(D)}_1(f_-) \geq Q^{(D)}_1(f_+,f_-) - Q^{(D)}_1(f_-) = Q^{(D)}_1(f,f_-)\\=\lim_{n \to \infty} Q^{(D)}_1(f,g_n)= \lim_{n \to \infty} \sum_{x \in X} (\cL  + 1)f(x) g_n(x) m(x)\geq 0.
\end{multline*}
This implies $Q^{(D)}_1(f_-) = 0$, showing $f_- = 0$.
\end{proof}

The following lemma is a slightly modified version of \cite[Lemma~6.6]{Soa}.

\begin{lemm}\label{lemma:does not vanish on remaining part of boundary}
 Let $F \subset \hat X \setminus \partial X$ be closed. Then there exists $g \in D(Q^{(D)}) \cap \ell^\infty(X)$ with $g\geq 0$  such that $g = 1$ on $F$. 
\end{lemm}
\begin{proof}
 Since $\hat X$ is compact, the set $F$ is also compact. By the definition of the $1$-harmonic boundary $\partial X$, for each $x \in F \setminus X $ there exists $g_x \in D(Q^{(D)}) \cap \ell^\infty(X)$ with $g_x(x) \neq 0$.  For each $x\in X$ there also exist such an $g_x$.  Using the contraction property of $Q_0$, we can assume $g_x \geq 0$. The family $\{g_x > 0\}$, $x \in F$, is an open cover of $F$. Hence, there exist finitely many $x_1,\ldots,x_n \in F$ such that  $g = \sum_{i = 1}^n g_{x_i} > 0$   on $F$. Using the compactness of $F$ again, we may rescale $g$ and assume $g \geq 1$ on $F$. The contraction property of $Q_0$ yields that $g\wedge 1 \in D(Q^{(D)}) \cap \ell^\infty(X)$ and it satisfies $g \wedge 1 = 1$ on $F$.  
\end{proof}

\begin{proof}[Proof of the maximum principle] With the help of the Royden decomposition, we write $f = f_h + f_0$ with a $1$-harmonic $f_h \in D(Q^{(N)})$ and $f_0 \in D(Q^{(D)}) $. Since $\cL f_0 + f_0 = \cL f + f \geq 0$, we infer $f_0 \geq 0$ from Lemma~\ref{lemma:positivity of superharmonic functions}. Moreover, since we assumed $ f\ge 0 $ and Proposition~\ref{proposition:unbounded extension} shows $f_0 = 0$ on $\partial X$, we infer $f_h \geq 0$ on $\partial  X$. Hence, it suffices to show the statement for $f_h$.

For   $n \in \N$ let $f_n = f \vee (-n)$. We show that the functions $(f_n)_h$ are nonnegative. Since $f_n \to f$, $n \to \infty$, with respect to $\av{\cdot}_{Q^{(N)}}$ and since the Royden decomposition is continuous, this implies $f_h \geq 0$.

Assume that $(f_n)_h$ is not nonnegative. Then, for all $\varepsilon > 0$ small enough, we have $F = \{x \in \hat X \mid (f_n)_h(x) \leq -\varepsilon\} \neq \emptyset$. By assumption we have $f \geq 0$ on $\partial X$. Since $f_n =  (f_n)_h$ on $\partial X$ and $f_n \geq f$ on $\partial X$, we infer $(f_n)_h \geq 0$ on $\partial X$, which implies $F \subset \hat X \setminus \partial X$. Using the previous lemma, we find $g \in D(Q^{(D)}) \cap \ell^\infty(X)$ with $g = 1$ on $F$.

Since  $(f_n)_h \geq -n$ (use Proposition~\ref{proposition:royden} and $f_n \geq -n$), we infer $(f_n)_h + n g \geq -\varepsilon$. Since $(f_n)_h$ is $1$-harmonic and $g \in D(Q^{(D)})$,  Proposition~\ref{proposition:royden} implies $-\varepsilon \leq ((f_n)_h + n g)_h = (f_n)_h$. Letting $\varepsilon \to 0+$ yields the claim.
%
%
%
%
%
\end{proof}

Next, we gather various consequences of the maximum principle.

\begin{corol}[Comparison principle]
 Assume that $f,g \in D(Q^{(N)})$ satisfy $\cL f + f \leq 0 \leq \cL g + g$ and $f \leq g$ on $\partial X$. Then $f \leq g$ on $X$. 
\end{corol}
\begin{proof}
The function $g - f$ is $1$-superharmonic and $g - f \geq 0$ on $\partial X$. Hence, the maximum principle implies $g - f \geq 0$ on $X$. 
\end{proof}

\begin{corol}[Kernel of $\gamma$] \label{coro:kernel gamma}
%
$\ker \gamma = D(Q^{(D)})$.
\end{corol}
\begin{proof} Proposition~\ref{proposition:unbounded extension} shows  
$$D(Q^{(D)})  \subset \{f \in D(Q^{(N)}) \mid f(z) = 0 \text{ for all } z \in  \partial X\} =\ker \gamma. $$
For the opposite inclusion, we let $f \in D(Q^{(N)})$ with $f = 0$ on $\partial X$ be given. By the Royden decomposition we can write $f = f_0 + f_h$ with $f_0 \in D(Q^{(D)})$ and $1$-harmonic $f \in D(Q^{(N)})$. Since $f_0 = 0$ on $\partial X$, we infer also $f_h = 0$ on $\partial X$. Hence, the maximum principle implies $f_h = 0$ and we obtain $f = f_0 \in D(Q^{(D)})$.  
%
%
%
%
\end{proof}

\begin{corol}[Dirichlet problem and $H\varphi$]
For any $\varphi \in {\rm ran}\, \gamma$, there exists a unique function $H\varphi \in D(Q^{(N)})$ solving the Dirichlet problem
$$\begin{cases}
   \cL H\varphi  + H \varphi  = 0 &\text{on } X,\\
   H\varphi  = \varphi &\text{on } \partial X.
  \end{cases}
$$
If $\varphi \geq 0$, then $H\varphi \geq 0$ and if $\varphi \leq 1$, then $H \varphi \leq 1$. 
\end{corol}
\begin{proof}
 Existence follows from the Royden decomposition and $\gamma(f) = \gamma(f_h)$. Uniqueness and the nonnegativity statement follow from the maximum principle. That $\varphi \leq 1$ implies $H\varphi \leq 1$ is a consequence of the comparison principle applied to the $1$-harmonic function $f = H \varphi$ and the $1$-superharmonic function $g = 1$.
\end{proof}

For $\varphi \in \ran \gamma$,   the unique solution to the Dirichlet problem  $H \varphi$ is called {\em $1$-harmonic extension of $\varphi$}. By definition it satisfies 
$$\gamma (H\varphi) = \varphi.$$

The $1$-harmonic extension is compatible with a variety of algebraic operations on $C(\partial X,\overline \R)$ (modulo $D(Q^{(D)})$). This is discussed in the next proposition.

\begin{propo}\label{proposition:harmonic extension and operations} Let $\varphi,\psi \in \ran \gamma$.
 \begin{enumerate}[(a)]
  \item If $C \colon \R \to \R$ is a monotone normal contraction, then $C\varphi \in \ran \gamma$ and  $C(H\varphi) - H (C\varphi) \in D(Q^{(D)})$ (where we tacitly extend $C$ to a map $\overline \R \to   \overline \R$ as discussed above). If $\varphi$ is bounded, then the statement holds for all normal contractions.
  \item $\varphi \wedge \psi \in \ran \gamma$ and $H (\varphi \wedge \psi) - (H \varphi) \wedge (H \psi) \in D(Q^{(D)})$.
  \item If $\varphi,\psi \in C(\partial X,\R)$, then $\varphi \psi \in \ran \gamma$ and $H(\varphi \psi) - (H\varphi) (H \psi) \in D(Q^{(D)})$.
 \end{enumerate}
\end{propo}
\begin{proof}
(a) Let $\varphi = \gamma(f)$ with $f \in D(Q^{(N)})$.  As discussed above, for a monotone normal contraction, we have $C(\gamma(f)) = \gamma (C(f))$, showing $C \varphi \in \ran \gamma$. We use the Royden decomposition to write  $C(H \varphi) = g_0 + g_h$ with $g_0 \in D(Q^{(D)})$ and $1$-harmonic $g_h \in D(Q^{(N)})$. Then, 
$$\gamma(g_h) = \gamma(C(H \varphi)) = C(\gamma(H\varphi)) = C \varphi = \gamma (H (C\varphi))$$
and the comparison principle imply $g_h = H (C\varphi)$, showing $C(H \varphi) - H (C\varphi) = g_0 \in D(Q^{(D)})$. For bounded $\varphi$, we can choose $f = H\varphi$, which is bounded. Then the same arguments as above yield the statement for arbitrary normal contractions.

(b) This can be inferred as in (a) using that $\gamma$ is a lattice homomorphism.

(c) Since $\partial X$ is compact (as a closed set in the compact space $\hat X$), the assumption implies that $\varphi,\psi$ are bounded. By the comparison principle also $H \varphi, H \psi$ are bounded. Since $D(Q^{(N)})\cap \ell^\infty(X)$ is an algebra, we obtain $(H \varphi)(H\psi) \in D(Q^{(N)}) \cap \ell^\infty(X)$ and we clearly have 
$$\gamma((H \varphi)(H\psi)) = \gamma(H \varphi) \gamma(H\psi) = \varphi \psi. $$
With this at hand the statement on the difference can be inferred as in (a). 
\end{proof}

\begin{rem}
 This proposition can be understood as saying that certain operations (lattice operations, compatibility with monotone normal contractions, multiplication) can be extended to the quotient space $ D(Q^{(N)})/D(Q^{(D)})$, which by the Royden decomposition theorem equals the space of  $1$-harmonic functions in $D(Q^{(N)})$.
\end{rem}

Before we continue, we have to discuss the range of $\gamma_0$, where we have to distinguish between the cases $1 \in D(Q^{(N)})$ and $1 \not \in D(Q^{(N)})$. Obviously, $1 \in D(Q^{(N)})$ if and only if 
$\sum_{x \in X} (c(x) + m(x)) < \infty.$
\begin{lemm}[Range of $\gamma_0$]\label{lemma:range of gamma}
 \begin{enumerate}[(a)]
  \item If $1 \in D(Q^{(N)})$, then ${\rm ran}\, \gamma_0$ is dense in $C(\partial X)$.
  \item  If $1 \not \in D(Q^{(N)})$, then there exists a point $\infty \in \partial X$ such that ${\rm ran}\, \gamma_0$ is dense in $\{f \in C(\partial X) \mid f(\infty)  = 0\}$, which is isometrically isomorphic to $C_0(\partial X \setminus \{\infty\})$.
 \end{enumerate}
\end{lemm}
\begin{proof}
 (a) This follows directly from the Stone-Weierstraß theorem. 
 
 (b) This follows directly from the Stone-Weierstraß theorem once we show the existence of $\infty  \in \partial X$ such that $f(\infty) = 0$ for all $f \in \ran \gamma_0$. Assume on the contrary that for all $x \in \partial X$ there exists $f_x \in D(Q^{(N)}) \cap \ell^\infty(X)$ such that $f_x(x) \neq 0$. Using Lemma~\ref{lemma:does not vanish on remaining part of boundary}, for each $x \in \hat X \setminus \partial X$ we find $f_{x} \in D(Q^{(N)}) \cap \ell^\infty(X)$ with $f_x(x) \neq 0$. Since $Q^{(N)}$ is a Dirichlet form, we can assume $f_x \geq 0$ with $f_x(x) > 1$ for all $x \in \hat X$ (else consider a scaled version of $|f_x|$). Using continuity, $U_x = \{y \in \hat X \mid f_x(y) > 1\}$, $x \in \hat X$, is an open cover of $\hat X$. By compactness of $\hat X$ it has a finite subcover $U_{x_1},\ldots,U_{x_n}$. We consider the function $f = \max\{f_{x_1},\ldots,f_{x_n}\}$. Since $Q^{(N)}$ is a Dirichlet form, we have $f \in D(Q^{(N)}) \cap \ell^\infty(X)$ and by construction it satisfies $f > 1$ on $\hat X$. Using again that $Q^{(N)}$ is a Dirichlet form, we infer $1 = f \wedge 1 \in D(Q^{(N)})$, a contradiction.  
\end{proof}

%

\begin{propo}[Harmonic measures]
 For each $x \in X$, there exists a Radon measure $\mu_x$  on $\partial X$ with $\mu_x(\partial X) \leq 1$  such that 
 $$H\varphi(x) = \int_{\partial X} \varphi \, d\mu_x, \quad \varphi \in {\rm ran}\, \gamma_0.$$
 If $1 \in D(Q^{(N)})$, then $\mu_x$ is unique and if $1 \not \in D(Q^{(N)})$, then $\mu_x$ is unique if we additionally assume $\mu_x(\{\infty\}) = 0$. Moreover, for $x,y \in X$ the measures $\mu_x$ and $\mu_y$ are mutually absolutely continuous with bounded Radon-Nikodym density.
\end{propo}
\begin{proof}
 Details for our setting are discussed in \cite[Proposition~2.1]{KLSS19}, see also \cite[Section~VI.4]{Soa}.
\end{proof}

The unique measures $\mu_x$, $x \in X$, (with the convention $\mu_x(\{\infty\}) = 0$ if $1 \not \in D(Q^{(N)})$) constructed in the previous proposition are called {\em $1$-harmonic measures}. 

\begin{rem}
Below we will fix a $1$-harmonic measure $\mu$ and use it as a reference measure on $\partial X$. Since the harmonic measures are equivalent, in principle our results do not depend on this choice. However,  in some formulas, different harmonic measures lead to different explicit expressions for the same object.   
\end{rem}

%
%
%
%
%
%
%
%

The following theorem, which will imply the continuity of the trace map, is due to Kasue \cite{Kas10} (for the Kuramochi boundary instead of the Royden boundary).

\begin{thm}[Kasue's theorem]
For all  $x \in X$, there exist $C_x \geq 0$ such that 
$$\int_{\partial X} |f|^2 d\mu_x \leq C_x \av{f}_{Q^{(N)}}^2, \quad \text{for all } f \in D(Q^{(N)}). $$
\end{thm}
\begin{proof}
For $f \in D(Q^{(N)}) \cap \ell^\infty(X)$, the statement is contained in \cite[Appendix~C]{KLSS19} (note that the form norm $\av{\cdot}_{Q^{(N)}}$ is larger than a constant times the norm considered there). For unbounded $f$, the statement follows from our approximation of $f$ via $(f \wedge n) \vee (-n)$ and the contraction property of $Q^{(N)}$. 
\end{proof}

\begin{rem}
This theorem shows that even though $f \in D(Q^{(N)})$ can attain the values $ \pm \infty$ on $\partial X$, it will only do so only on a $\mu_x$-null set. 
\end{rem}

From now on, we fix a $1$-harmonic measure $\mu$.
{For $1 \leq p \leq \infty$, we denote the space of $ p $-integrable functions on $\partial X$ with respect to $\mu$ by $\mathcal{L}^p(\partial X,\mu)$.}
Since different  $1$-harmonic measures are mutually absolutely continuous with bounded Radon Nikodym derivatives, for $1 \leq p \leq \infty$, the space $L^p(\partial X,\mu)$ and the norm topology on it is independent of the choice of $\mu$. Hence, we simply write $$ L^p(\partial X):=L^p(\partial X,\mu) $$ Note however, that the  norm on $L^p(\partial X)$ does depend on $\mu$. 

\begin{corol}[Continuity of the trace]\label{coro:continuity trace}
 The trace $\gamma \colon D(Q^{(N)}) \to L^2(\partial X)$ is continuous with respect to $\av{\cdot}_{Q^{(N)}}$ and $\ran \gamma$ is dense in $L^2(\partial X)$.
\end{corol}
\begin{proof}
 The continuity of the trace directly follows from Kasue's theorem. The density of $\ran \gamma$ follows either from $\ran \gamma_0 = C(\partial X)$ or from $\ran \gamma_0  = C(\partial X \setminus \{\infty\})$ and the assumption $\mu(\{\infty\}) = 0$. 
\end{proof}


%
%

%
%
%

\section{Dirichlet forms  and boundary conditions}\label{sec-Q}
As in the previous section, we fix a graph $(b,c)$ over $(X,m)$.   We  use the boundary theory developed in the previous section to describe various Markovian realizations of $\mathcal L$ (i.e. self-adjoint realization of $\mathcal L$ induced by Dirichlet forms) via abstract boundary conditions on $\partial X$. Specifically, we describe Markovian realizations arising from  Dirichlet forms  $Q$ that lie between $Q^{(D)}$ and $Q^{(N)}$ in the sense that $Q^{(N)}\leq Q \leq Q^{(D)}$ holds or, later, in the more restrictive sense of sandwiched semigroups.   The boundary conditions appear at first  abstractly as Markovian forms $q$ on $L^2(\partial X)$  such that $Q$ can be written as a sum of $Q^{(D)}$ or $Q^{(N)}$ and $q$.  Under additional assumptions the boundary conditions can also be expressed explicitly via the restrictions of the function to the boundary $\partial X$ and via a normal derivative on $\partial X$.

\subsection{Trace Dirichlet forms}
Let $Q$ be a Dirichlet form on $\ell^2(X,m)$ with 
$$ Q^{(N)} \leq    Q\leq Q^{(D)}. $$
We define the {\em trace form} ${\rm Tr}\, Q$ of $Q$ on $L^2(\partial X)$ as follows: Its domain is given by $D({\rm Tr}\, Q) = \gamma D(Q)$ and 
$${\rm Tr}\, Q(\varphi) = Q_1(H \varphi), \quad \varphi \in D({\rm Tr}\, Q).$$
The form  $$ q^{DN} = {\rm Tr}\, Q^{(N)} $$ is called {\em Dirichlet-to-Neumann form}.  

\begin{rem}
 We {highlight for} the reader that our trace Dirichlet forms are defined in terms of $Q_1 = Q + \av{\cdot}_{\ell^2(X,m)}^2$ and not in terms of $Q$. This is different from our source material \cite{KLSS19} because we work with $1$-harmonic extensions and not with harmonic extensions, see also Remark~\ref{remark: ell two vs measure free}. 
\end{rem}

Trace forms give representations of $Q$ via $Q^{(D)}$ and $Q^{(N)}$ as discussed in the subsequent lemma, cf. \cite[Lemma 3.11]{KLSS19}.

\begin{lemm}\label{lemma:q harmonicity}
 Let $Q$ be a Dirichlet form with $Q^{(N)} \leq Q \leq Q^{(D)}$. If $f = f_h + f_0 \in D(Q)$ with $f_0 \in D(Q^{(D)})$ and $1$-harmonic $f_h \in D(Q^{(N)})$, then $f_h \in D(Q)$ and 
 $$Q_1(f)  = Q_1^{(D)}(f_0) + Q_1(f_h).$$
In particular, 
$$Q_1 (f) =  Q_1^{(D)} (f_0) +  {\rm Tr}\, Q (\gamma f)$$  
and
$$Q (f) = Q^{(N)}(f) + {\rm Tr}\, Q (\gamma f) - q^{DN} (\gamma f)$$
\end{lemm}
\begin{proof} The inequality $Q^{(N)} \leq Q \leq Q^{(D)}$ together with $Q^{(D)} = Q^{(N)}$ on $D(Q^{(D)})$ 
 implies $D(Q^{(D)}) \subset D(Q)$ and $Q = Q^{(D)}$ on $D(Q^{(D)})$. In particular, $f_0 \in D(Q)$ and $f_h = f - f_0 \in D(Q)$ holds.  Using  $Q \geq Q^{(N)}$ and the $1$-harmonicity of $f_h$, we obtain 
\begin{multline*}
Q_1^{(D)}(f_0) + 2\alpha Q_1(f_0,f_h) + \alpha^2 Q_1(f_{h}) = Q_1(f_0 + \alpha f_h) \\
\geq Q_1^{(N)}(f_0 + \alpha f_h) = Q_1^{(D)}(f_0) + \alpha^2 Q^{(N)}_1(f_h). 
\end{multline*}
Subtracting $Q_1^{(D)}(f_0)$ on both sides, dividing by $\alpha$, {making the case distinction $ \alpha>0 $ and $ \alpha<0 $} and letting $\alpha \to 0\pm$, we find $Q_1(f_0,f_h) = 0$.
From this  and $Q = Q^{(D)}$ on $D(Q^{(D)})$, we derive
$$Q_1 (f) = Q_1 (f_0 + f_h) = Q_1 (f_0) + Q_1 (f_h) = Q_1^{(D)} (f_0) + Q_1 (f_h).$$
This is the first claimed equality.

From this and the definition of the trace form, we infer 
$$ Q_1 (f) =  Q_1^{(D)} (f_0) +  {\rm Tr}\, Q (\gamma f),$$
as $H(\gamma f) = f_h$  (solutions to the Dirichlet problem are unique and $\gamma f_h = \gamma f$ because $f_0$ vanishes on $\partial X$). Applying this formula with $Q =  Q^{(N)}$, we find
$Q_1^{(N)} (f) = Q_1^{(D)} (f_0) + q^{DN} (\gamma f)$, which leads to
$$Q_1^{(D)} (f_0) = Q_1^{(N)}(f) - q^{DN} (\gamma f).$$
Putting this in the first claimed equality gives  
$$Q (f) = Q^{(N)}(f) + {\rm Tr}\, Q (\gamma f) - q^{DN} (\gamma f)$$
and all claimed equalities are proven.  
\end{proof}

For the following proposition, we recall that a \emph{Dirichlet form in the wide sense} is a closed Markovian form which is not necessarily densely defined. 

\begin{propo}[Traces of Dirichlet forms]\label{theorem:traces}
 For every Dirichlet form $Q$ with  $ Q^{(N)} \leq    Q\leq Q^{(D)} $,
the trace form ${\rm Tr}\, Q$ is a Dirichlet form in the wide sense on $L^2(\partial X)$ and ${\rm Tr}\, Q - q^{DN}$ is Markovian (when considered as a quadratic form with  domain $D({\rm Tr}\, Q)$). Moreover, 
$$D(Q) = \{f \in D(Q^{(N)}) \mid \gamma f \in D({\rm Tr}\, Q)\}.$$
\end{propo}
\begin{proof} ${\rm Tr}\, Q$ is closed: Let $(\varphi_n)$ be $({\rm Tr}\, Q)_1$-Cauchy. Then there exists $h \in D(Q)$ and $\varphi \in L^2(\partial X)$ such that $H \varphi_n \to h$ with respect to $\av{\cdot}_{Q}$ and $\varphi_n \to \varphi$ with respect to $L^2(\partial X)$. Using the continuity of the trace map (Corollary~\ref{coro:continuity trace}) and $Q \geq Q^{(N)}$, we infer $\varphi_n = \gamma(H \varphi_n) \to \gamma(h)$ in $L^2(\partial X)$, i.e., $\varphi = \gamma(h)$.  Hence, $\varphi \in D({\rm Tr}\, Q)$ and $\varphi_n \to \varphi$ with respect to $({\rm Tr}\, Q)_1$.

 ${\rm Tr}\, Q$ is compatible with  normal contractions: Let $C$ be a normal contraction and let $\varphi \in D({\rm Tr}\, Q)$.

 Case 1.  $\varphi$ is bounded: Then  $C \varphi \in \gamma D(Q)$ and $ C(H\varphi) = H (C \varphi)  + f$ with $f \in D(Q^{(D)})$, see Proposition~\ref{proposition:harmonic extension and operations}. Using Lemma~\ref{lemma:q harmonicity}, this implies
\begin{multline*}
 {\rm Tr}\, Q(C\varphi) = Q_1(H (C \varphi)) \leq{ Q_1(H (C \varphi))+Q_1(f) }\\=Q_1(H (C \varphi) + f) 
 = Q_1(C(H \varphi)) \leq Q_1(H \varphi) = {\rm Tr}\, Q(\varphi).
\end{multline*}

Case 2. $C$ is increasing and $\varphi$ is arbitrary: The same arguments as in Case~1 yield the claim.

Case 3.  $C$ and $\varphi$ arbitrary:  Consider  the sequence of increasing normal contractions $D_n \colon \R \to \R$, $D_n(t) = (t \wedge n) \vee (-n)$. Using $C D_n(\varphi) \to C \varphi$, $n \to \infty$, in $L^2(\partial X$), and the boundedness of $D_n \varphi$, the lower semicontinuity of ${\rm Tr}\,Q$ and the already proven cases yield
$${\rm Tr}\,Q(C \varphi) \leq \liminf_{n\to \infty}{\rm Tr}\,Q(C  D_n\varphi) \leq \liminf_{n\to \infty}{\rm Tr}\,Q( D_n\varphi) \leq {\rm Tr}\, Q(\varphi). $$

${\rm Tr}\, Q - q^{DN}$ is Markovian: According to  \cite[Theorem~3.12]{KLSS19} the difference $Q_1 - Q^{(N)}_1$ is Markovian on $D(Q)$ (there extended forms are considered but the extended form of $Q_1$ equals $Q_1$). Hence, it follows with the same arguments as for ${\rm Tr}\, Q$  that ${\rm Tr}\, Q - q^{DN}$ is compatible with increasing normal contractions and with all normal contractions on bounded functions. Since it need not be lower semicontinuous, the approximation argument for unbounded functions is a bit more complicated (it relies on the form norm density of bounded functions in  $D({\rm Tr}\, Q)$, $D(q^{DN})$ and the fact that normal contractions act as form norm continuous (nonlinear) operators on domains of Dirichlet forms, see \cite{Anc76}).

For the ``moreover''-statement, we let $f \in D(Q^{(N)})$ with $\gamma f \in D({\rm Tr}\, Q)$ be given. Using the Royden decomposition, we write $f = f_0 + f_h$ with $f_0 \in D(Q^{(D)}) \subset D(Q)$ and $1$-harmonic $f_h \in D(Q^{(N)})$. By assumption there exists $g \in D(Q)$ such that $\gamma f = \gamma g$ and the previous lemma shows $g = g_0 + g_h$ with $g_0 \in D(Q^{(D)})$ and $1$-harmonic $g_h \in D(Q)$. Since $\gamma f_h = \gamma f = \gamma g = \gamma g_h$, the maximum principle yields $f_h  = g_h \in D(Q)$ and we arrive at $f = f_0 + f_h \in D(Q)$. 
\end{proof}

\begin{corol}[Representing Dirichlet forms via forms on the boundary] \label{coro:dirichlet forms via boundary} Let $Q$ be a Dirichlet form on $\ell^2(X,m)$ with
$ Q^{(N)} \leq    Q\leq Q^{(D)} $. Then, there exists a Dirichlet form in the wide sense  $q$ and a Markovian form $q'$ on  $L^2(\partial X)$ with 
$$Q_1 (f) = Q_1^{(D)} (f_0) + q (\gamma f) \mbox{ and } Q (f) = Q^{(N)}(f) + q'(\gamma f)$$
for all $f\in D(Q)$ with Royden decomposition $f = f_0 + f_h$. 
\end{corol}

%
%
%
%

\subsection{Sandwiched Dirichlet forms}
In the previous section we have seen that forms $Q$  with  $Q^{(N)} \leq Q \leq Q^{(D)}$ can be described by forms $q$  on the boundary. Here, we look at forms that lay  between $Q^{(N)}$ and $Q^{(D)}$ in an even stricter way. For such forms we can explicitly describe the arising forms on the boundary.

We fix a graph $(b,c)$ over $(X,m)$. We say that $Q$ is {\em sandwiched} between $Q^{(D)}$ and $Q^{(N)}$ if the associated semigroup $e^{-tL}$ satisfies $$ e^{-tL^{(D)}} f \leq e^{-tL}f \leq e^{-tL^{(N)}} f $$ for all nonnegative $f \in \ell^2(X,m)$ and $t \geq 0$. The following theorem characterizes sandwiched semigroups. In particular, it shows that forms sandwiched between $Q^{(D)}$ and $Q^{(N)}$ satisfy $Q^{(N)} \leq Q \leq Q^{(D)}$.
%
%

\begin{thm}[Arendt-Warma {theorem}]\label{theorem:arendt-warma}
Assume that $c = 0$, $m(X) < \infty$ and let $Q$ be a Dirichlet form on $\ell^2(X,m)$. The following assertions are equivalent: 
\begin{enumerate}[(i)]
 \item $Q$ is sandwiched between $Q^{(D)}$ and $Q^{(N)}$ and $D(Q) \cap C_c(X \cup \partial X)$ is dense in $D(Q)$ with respect to the form norm. 
 
 \item There exists a closed set $F \subset \partial X$ and a Radon measure $\nu$ on  $\partial X \setminus F$ such that 
 $$D(Q) = \{f \in D(Q^{(N)}) \mid f =  0 \text{ on } F \text{ and } \int_{\partial X \setminus F} |f|^2 d\nu < \infty\}$$
 and 
 $$Q(f) = Q^{(N)}(f) + \int_{\partial X \setminus F} |f|^2 d\nu.$$
\end{enumerate}
In particular, the generator of $Q$ is a restriction of $\cL$. 

\end{thm}

\begin{proof}
 This is basically the content of \cite[Corollary~4.7]{KLSSW23}. Here, we only explain how to translate the results (including how to address small technicalities) and why one assumption needed in \cite[Corollary~4.7]{KLSSW23} can be dropped in our case.  

 \cite[Corollary~4.7]{KLSSW23} characterizes all Dirichlet forms sandwiched between $Q^{(D)}$ and its active main part $(Q^{(D)})^{(M)}$ under the assumption that its killing part $(Q^{(D)})^{(k)}$ vanishes.  \cite[Example~3.10 and Example~3.16]{Schmi2} show that the active main part of $Q^{(D)}$ is given by
 $$D((Q^{(D)})^{(M)}) = \{f \in \ell^2(X,m) \mid \sum_{x,y \in X} b(x,y) (f(x) - f(y))^2 < \infty\} $$
 with 
 $$(Q^{(D)})^{(M)}(f) = {\frac{1}{2}}\sum_{x,y \in X} b(x,y) (f(x) - f(y))^2  $$
 and its killing part is given by $D((Q^{(D)})^{(k)})= D((Q^{(D)})^{(M)}) \cap \ell^2(X,c)$ with 
 $$(Q^{(D)})^{(k)}(f) = \sum_{x \in X}f(x)^2 c(x). $$
 In particular, $c = 0$ implies $(Q^{(D)})^{(k)} = 0$ and $Q^{(N)} = (Q^{(D)})^{(M)}$, so that indeed \cite[Corollary~4.7]{KLSSW23} characterizes all Dirichlet forms sandwiched between $Q^{(D)}$ and $Q^{(N)}$ via a representation as in (ii). 
 Hence, in principle, our theorem is proven. There are however some small technicalities, which upon closer inspection do not pose issues:
 \begin{itemize}
  \item The boundary used in \cite{KLSSW23} is smaller than the one we consider here (there is a continuous surjection of our boundary to the one in \cite{KLSSW23}), because the algebra used to construct the boundary in \cite{KLSSW23} can be strictly smaller than $D(Q^{(N)}) \cap \ell^\infty(X)$  (for technical reasons it is a a closure of a certain countably generated subalgebra of $D(Q^{(N)}) \cap \ell^\infty(X)$).  However, the proof given in \cite{KLSSW23}  works for our boundary as well.   
  \item In the representation in \cite[Corollary~4.7]{KLSSW23} the term $\int_{\partial X} |\tilde f|^2 d\nu$ appears, where $\tilde f$ is a quasi-continuous modification of $f \in D(Q)$. But $Q$ being sandwiched between $Q^{(D)}$ and $Q^{(N)} = (Q^{(D)})^{(M)}$ implies   $D(Q) \subset D((Q^{(D)})^{(M)}) = D(Q^{(N)})$, see \cite[Theorem~2.6]{KLSSW23}. In particular, any $f \in D(Q)$ is automatically continuous on $X \cup \partial X$. Therefore, we have a canonical quasi-continuous modification, namely the continuous extension of $f$ to $\partial X$. Moreover, in \cite{KLSSW23} the measure $\nu$ is assumed to be a  measure that does not charge sets of $Q^{(N)}$-capacity zero, because this implies that the integral $\int_{\partial X} |\tilde f|^2 d\nu$ does not depend on the choice of the quasi-continuous modification $\tilde f$ of $f$. Since we work with the continuous extensions, these regularity assumptions can be dropped in the proof of (ii) $\Rightarrow$ (i) of \cite[Corollary~4.7]{KLSSW23}.  
  \end{itemize}
  %
%
The ``in particular''-part follows from Green's formula and $g = 0$ on $\partial X$ if $g \in C_c(X)$.
%
\end{proof}

\begin{rem}
(a) In the case of Dirichlet forms generated by Laplacians on domains in Euclidean space the previous theorem  was first established by Arendt and Warma in \cite{AW03}.

(b) We assumed $m(X) < \infty$, because together with $c = 0$ it implies that ${\rm ran}\, \gamma_0$ is dense in $C(\partial X)$, which is crucial in the proof of \cite[Corollary~4.7]{KLSSW23}.  If $m(X) = \infty$, then all the results of the previous theorem hold true with $\partial X$ replaced by $\partial X \setminus \{\infty\}$, where $\infty$ is the point in $\partial X$ discussed in Lemma~\ref{lemma:does not vanish on remaining part of boundary}.
(c) The set $F$ and the measure $\nu$ (restricted to  $X \setminus F$) are uniquely determined up to sets of  $Q^{(M)}$-capacity zero, see \cite[Corollary~4.7]{KLSSW23}.
  
(d) In view of the decomposition discussed in Corollary~\ref{coro:dirichlet forms via boundary}, sandwiched Dirichlet forms correspond to Markovian forms $q'$ on $L^2(\partial X)$ that are induced by measures and come with a domain of functions vanishing on a closed subset of $\partial X$.
\end{rem}


\subsection{A normal derivative and boundary conditions} 
 In this subsection we define the normal derivative of a function as the functional (respectively function on the boundary), which makes Green's formula valid for functions not having compact support. With its help we describe several boundary conditions, leading to Markovian realizations of the discrete Laplacian.

For $f \in D(Q^{(N)})$ with $\cL f \in \ell^2(X,m)$, we introduce the functional 
$$\partial_n f \colon D(Q^{(N)}) \to \R, \quad \partial_n f (g) =  Q^{(N)}(f,g) - \as{\cL f,g}.$$
If $g \in C_c(X)$, then the previously discussed version of Green's formula yields $\partial_n f(g)=  0$. In comparison with Green's formula on open subsets of Euclidean space, $\partial_n f$  should be seen as a distributional version of the normal derivative of $f$.

Sometimes $\partial_n f$ can be interpreted as a function on $\partial X$. This is discussed next.  Let $\mu$ be a harmonic measure on $\partial X$.  By the density of $\ran \gamma$ in $L^2(\partial X)$, there exists at most one $\varphi \in L^2(\partial X)$ such that
$$\int_{\partial X} \varphi g d\mu = \partial_n f (g), \qquad g \in D(Q^{(N)}).$$
In this case, we call 
$$\partial_{n,\mu} f := \varphi$$
the {\em normal derivative} of $f$ with respect to $\mu$  and denote by $D(\partial_{n,\mu})$ the set of all $f \in D(Q^{(N)})$, whose normal derivative is given by a function in this sense. By the equivalence of the harmonic measures, the set $D(\partial_{n,\mu})$ does not depend on the choice of $\mu$ but the function $\partial_{n,\mu} f$ does. More precisely, if $\mu,\mu'$ are harmonic measures, then
$$\partial_{n,\mu'} f = \frac{d \mu}{d\mu'} \partial_{n,\mu} f.$$
We warn the reader that not for every $f \in D(Q^{(N)})$ with $\mathcal L f \in \ell^2(X,m)$ the normal derivative needs to be given by a function.

%

%

Next we introduce Robin boundary conditions via an auxiliary function $\beta$ on the boundary.  Let $\beta \colon \partial X \to [0,\infty]$ be measurable and let $$ F_\beta = {\rm supp}\, (1_{\{\beta = \infty\}}  \cdot \mu) ,$$ i.e., $x \in F_\beta$ if and only if $\mu(\{\beta = \infty\} \cap U) > 0$ for all open neighborhoods $U$ of $x$. We let $D(L^{\beta})$ be the set of all functions $f \in D(Q^{(N)})$ satisfying $\cL f \in \ell^2(X,m)$, $f \in \mathcal L^2(\partial X,\beta \mu)$ and  the boundary conditions

$$\begin{cases}
  f = 0  &\text{on } F_\beta,\\
(\beta + \partial_n) f = 0 &\text{on } \partial X \setminus F_\beta,
  \end{cases}
$$
where the equation $(\beta + \partial_n) f = 0$ on $\partial X \setminus F_\beta$ is to be understood as follows: For all $g \in D(Q^{(N)}) \cap \mathcal L^2(\partial X, \beta \mu)$ with $g = 0$ on $F_\beta$, the equality 
$$\partial_n f(g) +  \int_{\partial X} fg \beta d\mu = 0 $$
holds. We define the Laplacian with {\em generalized Robin boundary conditions} $L^{\beta}$ as the restriction of $\cL$ to the domain $D(L^{\beta})$.

\begin{rem}(a)   If  $\beta \in L^\infty(\partial X)$, $ \beta\ge0 $, then $F_\beta = \emptyset$ and $f \in  \mathcal L^2(\partial X,\beta \mu)$ for all $f \in D(Q^{(N)})$ follows from the continuity of the trace.  In this case,  the boundary condition simply becomes $f \in D(\partial_{n,\mu})$ (i.e. the trace of $f$ is given by a function) and $\beta f + \partial_{n,\mu} f = 0$. This can be seen as a discrete analogue of Robin boundary conditions.  In particular, for $\beta = 0$, we obtain Neumann boundary conditions.

 (b) If $\beta = \infty$, we obtain $F_\infty = \partial X$ and so any $f \in D(L^\infty)$ must satisfy $f = 0$ on $\partial X$. Moreover, any  $g \in D(Q^{(N)})$ with $g = 0$ on $F_\infty = \partial X$ belongs to $\ker \gamma = D(Q^{(D)})$. Using Green's formula for compactly supported functions and an approximation of $g \in D(Q^{(D)})$ with finitely supported functions yields $\partial_n f(g)  = 0$. Hence, the condition $(\beta + \partial_n) f = 0$ on $\partial X \setminus F_\infty$ already follows from $f = 0$ on $\partial X$  and we obtain
 $$D(L^\infty) = \{f \in D(Q^{(N)}) \mid \cL f \in \ell^2(X,m) \text{ with } f = 0 \text{ on } \partial X\}.$$
   This corresponds to Dirichlet boundary conditions.
\end{rem}
 
Next, we show that $L^\beta$ is a self-adjoint operator induced by a Dirichlet form. 

%

%

\begin{lemm}\label{lemma:robin}
 The quadratic form $Q^\beta$ defined by 
 $$D(Q^{\beta}) = \{f \in D(Q^{(N)}) \mid  \int_{\partial X} |f|^2 \beta d\mu < \infty\}$$
 and 
 $$Q^{\beta}(f) = Q^{(N)}(f) + \int_{\partial X} |f|^2 \beta d\mu$$
 is a Dirichlet form. Moreover, $f \in D(Q^{\beta})$ if and only if $f \in D(Q^{(N)})$ with $f  = 0$ on $F_\beta$ and $ \int_{\partial X \setminus F_\beta} |f|^2 \beta d\mu < \infty$.

\end{lemm}
\begin{proof}
 $Q^\beta$ is closed: Lower semicontinuity follows from standard arguments using the continuity of the trace and Fatou's lemma.
 
 $Q^\beta$ is Markovian: This is clear. 
 
 The formula for the domain: Clearly, it suffices to show the ``only if'' statement.   For $f \in D(Q^{\beta})$, it suffices to verify $f = 0$ on $F_\beta$. Let $x \in \partial X$ with $f(x) \neq 0$. Since $f$ is continuous, there exist $\varepsilon > 0$ and an open neighborhood $U$ of $x$ with $|f|^2 \geq \varepsilon$ on $U$. If $x \in F_\beta$, then by definition $\mu(U\cap \{\beta  = \infty\}) > 0$. We infer 
 $$\int_{\partial X} |f|^2 \beta d\mu \geq \int_{U \cap \{\beta = \infty\}} |f|^2 \beta d\mu \geq \varepsilon \mu(U \cap \{\beta = \infty\}) \cdot \infty = \infty,  $$
 a contradiction. Hence, $f = 0$ on $F_\beta$. 
\end{proof}

\begin{thm}[Robin boundary conditions]\label{theorem:robin}
 $L^{\beta}$ is the self-adjoint operator of the Dirichlet form $Q^\beta$, which was introduced in the previous lemma.
%
 %
 In particular, we have $Q^{(D)} = Q^\infty$ and $Q^{(N)} = Q^0$ and the corresponding operators $L^{(D)}$, $L^{(N)}$ are the restrictions of $\cL$ to
 $$D(L^{(D)}) = \{f \in D(Q^{(N)}) \mid \cL f \in \ell^2(X,m) \text{ with } f = 0 \text{ on } \partial X\}$$
 and
  \begin{align*}
  D(L^{(N)}) &= \{f \in D(Q^{(N)}) \mid \cL f \in \ell^2(X,m) \text{ with } \partial_n f = 0 \} = \ker \partial_{n,\mu}.
  \end{align*}
\end{thm}
\begin{proof}
 $L^\beta$ is the self-adjoint operator associated to the Dirichlet form $Q^\beta$: Let $L$ be the self-adjoint operator induced by $Q^\beta$.  For $f \in D(L)$, we have
 $$\as{Lf, g} = Q^\beta(f,g) = Q^{(N)}(f,g) + \int_{\partial X} fg \beta d\mu$$
 for all $g \in D(Q^\beta)$. Since $C_c(X) \subset D(Q^\beta) \cap \ker \gamma$, this computation and Green's formula yield $  \cL f = Lf \in \ell^2(X,m)$. By the previous lemma we also have $f = 0$ on $F_\beta$. Any $g \in D(Q^{(N)}) \cap \mathcal L^2(\partial X,\beta \mu)$ with $g = 0$ on $F_\beta$ belongs to $D(Q^\beta)$ and for such $g$ we obtain 
 $$\partial_n(f)(g) =  Q^{(N)}(f,g)  -  \as{\cL f, g} = Q^{(N)}(f,g) - Q^\beta(f,g) = -  \int_{\partial X} fg \beta d\mu, $$
 showing $(\partial_n  + \beta) f = 0$ on $X \setminus F_\beta$  (in the sense of our definition).
 
 Conversely, assume that $f \in D(L^\beta)$. The assumption $f \in \mathcal L^2(\partial X,\beta\mu)$ implies $f \in D(Q^\beta)$. Moreover, by the previous lemma any $g \in D(Q^\beta)$ satisfies $g = 0$ on $F_\beta$ and $g \in \mathcal L^2(\partial X,\beta \mu)$. Hence, the assumption $(\partial_n + \beta) f = 0$ on $X \setminus F_\beta$ implies 
 $$\partial_n(f)(g) = - \int_\beta fg \beta d\mu$$
 for all $g \in D(Q^\beta)$. Spelled out, this implies $\cL f \in \ell^2(X,m)$ and 
 $$Q^\beta(f,g) = \as{\cL f,g}$$
 for all $g \in D(Q^\beta)$, which yields $f \in D(L)$.

 $Q^{(N)} = Q^0$ follows directly from the definitions and $Q^{(D)} = Q^\infty$ is a consequence of Corollary~\ref{coro:kernel gamma}. The statements on the corresponding operators follow from the definition of $L^\beta$ and $\ker \partial_n = \ker \partial_{n,\mu}$.
\end{proof}


\begin{corol}[Sandwiched forms and Robin boundary conditions]
Assume that $c = 0$ and $m(X) < \infty$. 
\begin{enumerate}[(a)]
 \item Let $\beta \colon \partial X \to [0,\infty]$ be measurable with $\beta \in \mathcal L^1_{\rm loc}(\partial X \setminus F_\beta)$. Then $Q^\beta$ is sandwiched between $Q^{(D)}$ and $Q^{(N)}$.
 \item Let $Q$ be a Dirichlet form on $\ell^2(X,m)$ that is sandwiched between $Q^{(D)}$ and $Q^{(N)}$ and assume that  $D(Q) \cap C_c(X \cup \partial X)$ is dense in $D(Q)$ with respect to the form norm. Moreover, let $L$ be the self-adjoint generator of $Q$. If there exists $u \in D(L) \cap D(\partial_{n,\mu})$ such that $u > 0$ on $\partial X$, then  there exists a measurable $\beta \colon \partial X \to [0,\infty)$ with $\beta \in \mathcal L^1(\partial X)$ such that $Q = Q^\beta$ and $L = L^\beta$.
\end{enumerate}
\end{corol}
\begin{proof}
 (a)  Due to the integrability assumption on $\beta$, $\beta \mu$ is a Radon measure on $\partial X \setminus F_\beta$. Hence, the statement follows from Theorem~\ref{theorem:arendt-warma} and our characterization of the form domain of $Q^\beta$ in Lemma~\ref{lemma:robin}. 
 
 (b) According to Theorem~\ref{theorem:arendt-warma} there exists a closed set $F \subset \partial X$ and a Radon measure $\nu$ on $\partial X \setminus F$ such that
 $$D(Q) = \{f \in D(Q^{(N)}) \mid f =  0 \text{ on } F \text{ and } \int_{\partial X \setminus F} |f|^2 d\nu < \infty\}$$
 and 
 $$Q(f) = Q^{(N)}(f) + \int_{\partial X} |f|^2 d\nu, \quad f \in D(Q).$$
 Since $u \in D(Q)$ and $u  > 0$ on $\partial X$, we infer $F = \emptyset$. Moreover, the compactness of $\partial X$ and the continuity of $u$ imply $\inf_{\partial X} u > 0$, which yields that $\nu$ is a finite measure. The description of $D(Q)$ and the finiteness of $\nu$ imply $D(Q^{(N)}) \cap \ell^\infty(X) \subset D(Q)$. Using $u \in D(L) \cap D(\partial_{n,\mu})$ and that $L$ is a restriction of $\cL$, for each $f \in D(Q^{(N)}) \cap \ell^\infty(X)$, we infer 
 \begin{align*}
\int_{\partial X} \partial_{n,\mu} u f d\mu =  Q^{(N)}(u,f) -  \as{Lu,f} = Q^{(N)}(u,f) - Q(u,f) = - \int_{\partial X} u f d\nu. 
 \end{align*}
Since $\mu$ and $\nu$ are finite measures, $u \in L^2(\partial X,\nu)$ and  $\partial_{n,\mu} u \in L^2(\partial X)$,  ${\rm ran}\, \gamma_0$ is dense in $C(\partial X)$ and $u > 0$, we infer 
$$\nu = - \frac{\partial_{n,\mu} u}{u} \mu.$$
This implies the statement with $\beta = - (\partial_{n,\mu} u) / u$.
\end{proof}

\section{Dirichlet forms via forms on the boundary}\label{ref-allDF}
In this section we address the question which Dirichlet forms exist on discrete spaces and provide answers both in terms of graphs and of Laplacians.  To set our results in perspective we note that there is a one-to-one correspondence between regular Dirichlet forms and graphs. Indeed, the regular Dirichlet forms on $(X,m)$ are exactly the forms  $Q^{(D)}_{b,c}$, see Theorem~\ref{IG:t:graphs_regular_df}, which is taken from \cite{KL}. This then also gives that any Dirichlet form $Q$ on $(X,m)$ with $C_c (X)\subset  D(Q)$ is an extension of the form $Q^{(D)}_{b,c}$ for a graph $(b,c)$ (as $Q$  is an extension of the closure of its restriction to $C_c (X)$ and this restriction is a  regular Dirichlet form).

Let $(b,c)$ be a graph over $(X,m)$ with associated  $1$-harmonic boundary $\partial X$. We write $\mathfrak{D}(b,c)$ for the set of all Dirichlet forms on $\ell^2(X,m)$ satisfying $Q^{(N)} \leq Q \leq Q^{(D)}$ and we denote by  ${\rm Tr}\, \mathfrak{D}(b,c)$ the set of all Dirichlet forms in the wide sense $q$ on $L^2(\partial X)$ for which $q \geq q^{DN}$ and $q - q^{DN}$ is Markovian (when considered as a quadratic form on the domain  $D(q)$).

\begin{thm} \label{theorem:dirichlet forms associated to a graph}
 Let $(b,c)$ be a graph over $(X,m)$ with associated  $1$-harmonic boundary $\partial X$. For a closed quadratic form $Q$ on $\ell^2(X,m)$, the following assertions are equivalent. 
 \begin{enumerate}[(i)]
  \item $Q \in \mathfrak{D}(b,c)$.
  \item $D(Q) \subset D(Q^{(N)})$ and  there exists $q \in {\rm Tr}\, \mathfrak{D}(b,c)$ such that for all $f \in D(Q)$ we have $\gamma f \in D(q)$ and
  $$ Q(f) = Q^{(N)}(f) + q(\gamma f) - q^{DN}(\gamma f),$$
  as well as 
  $$Q_1(f) =  Q_1^{(D)}(f_0) + q(\gamma f).$$
  Here, $f = f_0 + f_h$ with $f_0 \in D(Q^{(D)})$ and $1$-harmonic $f_h \in D(Q^{(N)})$ is the Royden decomposition of $f$.    
 \end{enumerate}
 More precisely, if $Q \in \mathfrak{D}(b,c)$, then  in (ii) the form $q$ can be chosen as $q = {\rm Tr}\, Q$ and the map
 $$\mathfrak{D}(b,c) \to  {\rm Tr}\, \mathfrak{D}(b,c), \quad Q \mapsto {\rm Tr}\, Q $$
 is a bijection. 
\end{thm}
\begin{proof}
(ii) $\Rightarrow$ (i): The Markov property of $Q^{(N)}$, $q - q^{DN}$ and that the trace map commutes with monotone normal contractions imply that $Q$ is Markovian. Moreover, $Q \geq Q^{(N)}$ follows from $q \geq q^{DN}$ and $Q^{(D)} \geq Q$ follows from $\ker \gamma = D(Q^{(D)})$.

 (i) $\Rightarrow$ (ii) and $q$ can be chosen as $q = {\rm Tr}\, Q$: We saw in Proposition~\ref{theorem:traces} that $q = {\rm Tr}\, Q$ belongs to ${\rm Tr}\, \mathcal{D}(b,c)$. Moreover, for $f \in D(Q)$ Lemma~\ref{lemma:q harmonicity} implies $Q_1(f) = Q^{(D)}_1(f_0) + Q_1(f_h)$.  Using this formula and $f_h = H (\gamma f)$, we directly infer the second formula for $Q$ and 
 \begin{align*}
 Q(f) &= Q^{(N)}(f) + Q_1(f)  - Q^{(N)}_1(f) \\
 &= Q^{(N)}(f) + Q^{(D)}_1(f_0) + Q_1(f_h)  - Q^{(D)}_1(f_0) - Q^{(N)}_1(f_h)\\
 &= Q^{(N)}(f) + {\rm Tr}\, Q(\gamma f) - q^{DN}(\gamma f).  
 \end{align*}

 It remains to show the bijectivity of the trace map on forms. 
 
 Injectivity: Let $Q,Q' \in \mathfrak{D} (b,c)$ with ${\rm Tr}\, Q =  {\rm Tr}\, Q'$. For $f \in D(Q)$, we then obtain $f \in D({\rm Tr}\,Q) = D({\rm Tr}\,Q')$ and Proposition~\ref{theorem:traces} yields $f \in D(Q')$. With this at hand, $Q(f)  = Q'(f)$ follows from the already established formula for $Q,Q'$ in terms of their traces. Interchanging the role of $Q$ and $Q'$ yields $Q = Q'$. 
 
 Surjectivity: We let $q \in {\rm Tr}\, \mathfrak{D}(b,c)$ and define the quadratic form $Q$ by $D(Q) = \{f \in D(Q^{(N)}) \mid \gamma f \in D(q)\}$ and 
 $$Q(f) = Q^{(N)}(f) + q(\gamma f) - q^{DN}(\gamma f). $$
 For $f \in D(Q)$ the definition of $q^{DN}$ yields $Q_1(f) = Q^{(D)}_1(f_0) + q(\gamma f)$.
 
 $Q$ has the Markov property: This follows from the Markov property of $Q^{(N)}$ and $q - q^{DN}$. 
 
 $Q$ is closed: Let $(f_n)$ be $Q_1$-Cauchy. Since $q \geq q^{DN}$,  $(f_n)$ is also $Q^{(N)}_1$-Cauchy and the closedness of $Q^{(N)}$ implies $f_n \to f$ for some $f \in D(Q^{(N)})$ with respect to $Q^{(N)}_1$. The continuity of the trace map implies $\gamma f_n \to \gamma f$ in $L^2(\partial X)$ and the continuity of the Royden decomposition yields $(f_n)_0 \to f_0$ with respect to $Q^{(D)}_1$.  Moreover, the second formula for $Q$ shows that $(\gamma f_n)$ is also $q$-Cauchy and we obtain $\gamma f_n \to \gamma f$ with respect to $q_1$ from the closedness of $q$. Overall, this implies
 $$Q_1(f - f_n) = Q_1^{(D)} (f_0 - (f_n)_0) + q(\gamma f_n - \gamma f) \to 0, \text{ as } n \to \infty,$$
 showing $f_n \to f$ with respect to $Q_1$.

 ${\rm Tr}\, Q = q$: By definition we have $D({\rm Tr}\, Q) \subset D(q)$. For the reverse inclusion, we note that  $q \geq q^{DN}$ implies $D(q) \subset D(q^{DN}) = \ran \gamma$. Hence, for $\varphi \in D(q)$, there exists $f \in D(Q^{(N)})$ with  $\varphi = \gamma f$, showing $\varphi \in D({\rm Tr}\, Q)$. Moreover, for $\varphi \in {\rm Tr}\, Q$, we infer 
 \begin{align*}
& {\rm Tr}\, Q(\varphi) = Q_1(H \varphi) = Q^{(D)}_1((H\varphi)_0)  + q(\gamma (H\varphi)) = q(\varphi),
 \end{align*}
 proving ${\rm Tr}\, Q = q$.
\end{proof}

  We denote by by $\mathfrak{D}$ the collection of all Dirichlet forms $Q$ on $\ell^2(X,m)$ whose domain satisfies $C_c(X) \subset D(Q)$. With the help of the previous theorem, we can characterize all Dirichlet forms in $\mathfrak{D}$ without potential term through graphs without potential and Dirichlet forms on their boundary.   In the following corollary various graphs play a role. We use subscripts to indicate with respect to which graph the objects are defined.

\begin{corol}[Dirichlet forms on discrete spaces without potential] \label{coro:discrete dirichlet forms without killing}{Let $ (X,m) $ be a discrete measure space.}
 For $Q \in \mathfrak{D}$, the following assertions are equivalent.
 \begin{enumerate}[(i)]
  \item There exists a graph $(b_Q,0)$ such that $Q$ is an extension of $Q^{(D)}_{b_Q,0}$.
  \item  There exists a  graph $(b_Q,0)$  such that $D(Q) \subset D(Q^{(N)}_{b_Q,0})$ and  there exists $q \in {\rm Tr}\, \mathfrak{D} (b_Q,0)$ such that
     $$Q(f) = Q_{b_Q,0}^{(N)}(f) + q(\gamma f) - q^{DN}_{b_Q,0}(\gamma f), \quad f \in D(Q).$$
 \end{enumerate}
  More precisely, (i) or (ii) determine the same unique graph $(b_Q,0)$ and the form $q$ in (ii) can be chosen as $q = {\rm Tr}\, Q$  (where the trace is taken with respect to the $1$-harmonic boundary of $(b_Q,0)$).
  
  If we denote by $\mathfrak{D}_0$ the Dirichlet forms characterized by (i) and (ii), then the map
 \begin{gather*}
 \mathfrak{D}_0 \to \bigsqcup_{(b,0) \text{ graph}}  {\rm Tr}\, \mathfrak{D}(b,0)\\ Q \mapsto ((b_Q,0), {\rm Tr}\, Q)
 \end{gather*}
 is a bijection.
\end{corol}

\begin{proof}
(ii) $\Rightarrow$ (i): This is trivial {since $ C_{c}(X)\subset  {D}(Q) $ by definition of $ \mathfrak{D} $}.

(i) $\Rightarrow$ (ii) and the ``more precisely''-statement: If $Q$ is an extension of $Q^{(D)}_{b_Q,0}$, then \cite[Lemma~6.7]{Sch20} shows that for finite $K \subset X$ and $f \in D(Q) \cap \ell^\infty(X)$ we have 
\begin{multline*}
Q(f) \geq Q(1_K f) - Q(1_K f^2,1_K) = Q^{(D)}_{b_Q,0}(1_K f) - Q^{(D)}_{b_Q,0}(1_K f^2, 1_K)\\= \frac{1}{2}\sum_{x,y \in K} b(x,y)(f(x) - f(y))^2. 
\end{multline*}
The last equality is a simple computation, see e.g. \cite[Proposition~3.18]{KLW21}. Letting $K \nearrow X$ implies $Q(f) \geq Q^{(N)}_{b_Q,0}(f)$, and since bounded functions are dense in the domain of Dirichlet forms, we arrive at $Q \geq Q^{(N)}_{b_Q,0}$.

With this at hand, (ii) and that $q$ can be chosen as the trace of $Q$ follow from the previous theorem. The formula for the graph and its uniqueness are a consequence of the formula
$$b_Q(x,y) = -Q^{(D)}_{(b_Q,0)}(1_{\{x\}},1_{\{y\}}) = -Q(1_{\{x\}},1_{\{y\}}),$$ 
for $x,y \in X$ with $x \neq y$.

The uniqueness of $(b_Q,0)$ and the previous theorem imply the bijectivity of the map in the last statement. 
\end{proof}

\begin{rem}   For any Dirichlet form $Q \in \mathfrak{D}$, there exists a unique graph $(b,c)$ such that $Q$ is an extension of $Q^{(D)}_{b,c}$, see Theorem~\ref{IG:t:graphs_regular_df}.  Hence, the previous corollary characterizes precisely those Dirichlet forms in $\mathfrak{D}$ with $c = 0$ -- the discrete Dirichlet forms without killing.   Unfortunately, without assuming $c = 0$, the result does not hold anymore, see the next example. One can amend this by making assumptions on the action of the generator, see the next corollary.  
\end{rem}

\begin{exa}
We consider the  graph $\Z^3$ {with standard weights} (i.e., $X = \Z^3$ and $b(x,y) = 1$ if and only if $|x-y|_1 = 1$) and choose a finite measure $m \colon \Z^3 \to (0,\infty)$. Then   $1 \not \in  D(Q^{(D)}_{b,0})$, $D(Q^{(D)}_{b,0}) \subset C_0(\Z^3)$,  {where $ C_{0}(\Z^{3}) $ is the uniform closure of ${C_{c}(\Z^3)}$} and
$$D(Q^{(N)}_{b,0}) = D(Q^{(D)}_{b,0}) \oplus \R \cdot 1,$$
where the sum is orthogonal with respect to the semi-scalar product $Q^{(N)}_{b,0}$,  see e.g. the example in \cite[Section~6]{KLSW17}. In particular, for each $f \in D(Q^{(N)}_{b,0})$ the limit 
$$f(\infty) = \lim_{|x| \to \infty} f(x)$$
exists, the compactification $\hat \Z^3 = \Z^3 \cup \{\infty\}$ is the one-point compactification of $\Z^3$ and the $1$-harmonic measure $\mu$ on $\{\infty\}$ is  a  multiple of the Dirac measure. We consider the Dirichlet form $Q$ on $\ell^2(X,m)$ with domain $D(Q) = D(Q^{(N)}_{b,0})$ and
$$Q(f) = Q^{(N)}_{b,0}(f) + |f(0) - f(\infty)|^2.$$
For no graph $(\tilde b,\tilde c)$ and no Markovian quadratic form $q$ on the $1$-harmonic boundary of $(\tilde b,\tilde c)$, the equality  
$$Q(f) = Q^{(N)}_{\tilde b,\tilde c} (f) + q(\gamma f)$$
can hold for all $f \in D(Q)$. Indeed, if this equality were true for some graph $(\tilde b,\tilde c)$ and some form on the boundary, then its validity for all $f \in C_c(X)$ would imply  $b = \tilde b$ and $c = \tilde c$. But we have 
$$Q(1) =  0 < 1 = c(0) = Q^{(N)}_{b,c}(1) \leq Q^{(N)}_{b,c}(1) + q(\gamma 1),$$ 
a contradiction.
\end{exa}

\begin{corol}[Dirichlet forms associated to discrete Laplacians] \label{coro:discrete dirichlet forms with discrete laplacian} {Let $ (X,m) $ be a discrete measure space.}
  For $Q \in \mathfrak{D}$, the following assertions are equivalent. 
 \begin{enumerate}[(i)]
  \item There exists a graph $(b_Q,c_Q)$ such that the self-adjoint operator of $Q$  is a restriction of $\mathcal L_{b_Q,c_Q}$.
  \item  There exists a  graph $(b_Q,c_Q)$ over $(X,m)$ such that $D(Q) \subset D(Q^{(N)}_{b_Q,c_Q})$ and  there exists $q \in {\rm Tr}\, \mathfrak{D} (b_Q,c_Q)$ such that 
     $$Q(f) = Q_{b_Q,c_Q}^{(N)}(f) + q(\gamma f) - q^{DN}_{b_Q,0}(\gamma f), \quad f \in D(Q).$$
 \end{enumerate}
  More precisely, (i) or (ii) determine the same unique graph $(b_Q,c_Q)$  and the form $q$ in (ii) can be chosen as $q = {\rm Tr}\, Q$  (where the trace is taken with respect to the $1$-harmonic boundary of $(b_Q,c_Q)$).
  
  If we denote by $\mathfrak{D}_g$ the Dirichlet forms characterized by (i) and (ii), then the map
 \begin{gather*}
 \mathfrak{D}_g \to \bigsqcup_{(b,c) \text{ graph}}  {\rm Tr}\, \mathfrak{D}(b,c)\\ Q \mapsto ((b_Q,c_Q), {\rm Tr}\, Q)
 \end{gather*}
 is a bijection.
\end{corol}
 \begin{proof} (ii) $\Rightarrow$ (i): This is a consequence of Green's formula and $\gamma f =  0$ for all $f \in C_c(X)$. 
 
 (i) $\Rightarrow$ (ii):  \cite[Theorem~6.5]{Sch20} shows that any Dirichlet form $Q$ with $C_c(X) \subset D(Q)$ whose generator is a restriction of $\mathcal L_{b,c}$ satisfies $ Q^{(D)}_{b,c} \leq Q \leq Q^{(N)}_{b,c} $ and so the statement follows from the previous theorem.
 
 The remaining bijectivity statement can be inferred as in the previous corollary. 
 \end{proof}

 \begin{rem}
  Theorem~\ref{theorem:dirichlet forms associated to a graph} is a streamlined version of \cite[Theorem~3.5. and Corollary~3.6.]{KLSS19}, which are the main results of this paper. Corollary~\ref{coro:discrete dirichlet forms without killing} and Corollary~\ref{coro:discrete dirichlet forms with discrete laplacian} use this theorem to characterize almost all Dirichlet forms whose domain contains the ``discrete test functions'' $C_c(X)$. 
 \end{rem}

%
%
%

\textbf{Acknowledgements.} The authors appreciate the financial support of the DFG within the priority program ``Geometry at infinity''. Moreover, we express our gratitude to Bobo Hua, Xueping Huang, Simon Puchert, Michael Schwarz und Melchior Wirth who coauthored the works which are the foundation of this survey. Finally, we would also like to thank the ``anonymous'' referee Kai-Uwe Bux for his careful reading and  helpful comments.

\printbibliography
\end{document}